\documentclass[a4paper,11pt]{amsart}

\usepackage{amssymb,amsthm}  
\usepackage{mathtools}      
\usepackage{mathabx}        
\usepackage[bb=fourier,cal=euler,scr=rsfs]{mathalfa}	
\usepackage{enumitem}       
\usepackage[ansinew]{inputenc}
\usepackage{color}
\usepackage{tikz}
\usepackage{upref}
\usepackage{graphicx}

\numberwithin{equation}{section}     

\setlist[enumerate,1]{label={\upshape(\roman*)},ref=\roman*}
\setlist[enumerate,2]{label={\upshape(\alph*)},ref=\alph*}

   \def\DD{{\mathbb D}}

 \def\NN{{\mathbb N}}  
 \def\RR{{\mathbb R}}  \def\TT{{\mathbb T}}
   
 \def\ZZ{{\mathbb Z}}

\def\cA{\mathcal{A}}  \def\cG{\mathcal{G}}  \def\cS{\mathcal{S}}
   \def\cN{\mathcal{N}} \def\cT{\mathcal{T}}
\def\cC{\mathcal{C}}   \def\cO{\mathcal{O}} \def\cU{\mathcal{U}}
\def\cD{\mathcal{D}}    \def\cV{\mathcal{V}}
\def\cE{\mathcal{E}}    \def\cW{\mathcal{W}}

\newtheorem*{teo*}{Theorem}

\newtheorem*{teoA}{Theorem A}
\newtheorem*{teoB}{Theorem B}

\newtheorem*{teoC}{Theorem C}
\newtheorem{teo}{Theorem}[section]

\newtheorem*{conj*}{Conjecture}

\newtheorem{cor}[teo]{Corollary}
\newtheorem*{corD}{Corollary D}

\newtheorem{lema}[teo]{Lemma}
\newtheorem{prop}[teo]{Proposition}

\newcommand{\bi}{\begin{itemize}}
\newcommand{\ei}{\end{itemize}}

\theoremstyle{definition}

\theoremstyle{remark}
\newtheorem{obs}[teo]{Remark}

\newcommand{\eps}{\varepsilon}
\DeclareMathOperator{\interior}{int}

\newcommand{\en}{\subset}

\DeclareMathOperator{\Diff}{Diff}
\DeclareMathOperator{\PH}{PH}

\DeclareMathOperator{\Image}{Image}
\newcommand{\diff}[1]{\Diff^{#1}(M)}

\author[S. Crovisier]{Sylvain Crovisier}

\author[R. Potrie]{Rafael Potrie}

\author[M. Sambarino]{Mart\'in Sambarino}

\title[Finiteness of partially hyperbolic attractors]{Finiteness of partially hyperbolic attractors with one-dimensional center}

\date{\today}

\thanks{The authors were partially supported by the Balzan Research Project of J. Palis.
R.P. and M.S were partially supported by CSIC group 618, IFUM, CNRS and MathAmSud:Physeco.
R.P. was also partially supported by the Laboratoire Mathematique d'Orsay. S. C. was partially supported by IFUM and the ERC project 692925 \emph{NUHGD}. }

\begin{document}

\maketitle

\begin{abstract}
We prove that the set of diffeomorphisms having at most finitely many attractors
contains a dense and open subset of the space of $C^1$ partially hyperbolic diffeomorphisms with one-di\-men\-sio\-nal center.

This is obtained thanks to a robust geometric property of the stable and unstable laminations that we show to hold after perturbations of the dynamics.
This technique also allows to prove that $C^1$-generic diffeomorphisms far from homoclinic tangencies in dimension $3$
either have at most finitely many attractors, or satisfy Newhouse phenomenon.
\bigskip

\noindent
{\bf Finitude des attracteurs partiellement hyperboliques avec central de dimension un.}

\noindent
{\sc R\'esum\'e.}
Nous montrons que l'ensemble des diff\'eomorphismes ayant un nombre au plus fini d'attracteurs contient un ouvert dense de l'espace
des diff\'eomorphismes $C^1$ partiellement hyperboliques avec fibr\'e central de dimension $1$.

Ce r\'esultat d\'ecoule d'une propri\'et\'e g\'eom\'etrique robuste des laminations stables et instables,
qui peut \^etre obtenue par perturbation de la dynamique.
Cette technique nous permet \'egalement de montrer que sur les vari\'et\'es de dimension $3$,
les diff\'eomorphismes $C^1$-g\'en\'eriques loin des tangences homoclines ou bien ont un nombre au plus fini
d'attracteurs, ou bien pr\'esentent le ph\'enom\`ene de Newhouse.
\bigskip


\medskip

\end{abstract}

\section{Introduction}\label{SectionIntroduction}

A main question when one studies the qualitative properties of a dynamical system consists in describing its attractors.
More generally, one studies how the dynamics decomposes into elementary invariant pieces. This is for instance the purpose
of Smale's spectral decomposition theorem for hyperbolic dynamics.
This paper discusses the number of attractors for diffeomorphisms $f$ of a compact boundaryless manifold
under a weaker hyperbolicity property.

One usually defines an \emph{attractor} of $f$ as an $f$-invariant non-empty compact set $K$ which admits a neighborhood $U$ satisfying $K=\bigcap_{n\in \NN} f^n(U)$ and which is transitive
(i.e. the dynamics of $f$ on $K$ contains a dense forward orbit).
An attractor which is reduced to a finite set is called a \emph{sink}.
In general a diffeomorphism may have no attractors (this is for instance the case of the identity)
and one introduces a weaker notion: a \emph{quasi-attractor} of $f$ is a $f$-invariant non-empty compact set which has the following two properties:
\begin{itemize}
\item $K$ admits a basis of open neighborhoods $U$ such that $f(\overline U)\subset U$,
\item $K$ is chain-transitive, i.e. for any $\varepsilon>0$ there exists a dense sequence $(x_n)_{n\geq 0}$ in $K$
which satisfies $d(f(x_n),x_{n+1})<\varepsilon$ for each $n\geq 0$.
\end{itemize}
Any homeomorphism of a compact metric space admits at least one quasi-attractor.
For hyperbolic diffeomorphisms they coincide with usual attractors.
For any diffeomorphisms in a dense G$_{\delta}$-set of $\Diff^1(M)$,
the set of points whose positive orbit accumulate on a quasi-attractor is
a dense G$_\delta$-subset of $M$, see~\cite{BC-rec}.

The number of attractors may be infinite for large classes of dynamical systems. This is the case near the set $\cT$
of diffeomorphisms exhibiting a homoclinic tangency, i.e. which have a hyperbolic periodic orbit whose stable and unstable manifolds
are not transverse: this has been proved by Newhouse~\cite{newhouse1}
inside the space $\Diff^2(M)$ of $C^2$ diffeomorphisms of a surface $M$,
or in $\Diff^1(M)$ when $\dim(M)\geq 3$, under a stronger assumption on the homoclinic
tangency, see for instance~\cite{BD,BDV,Bonatti-Survey,Crov-habilitation}.
In fact, all the known abundant classes of diffeomorphisms are in the limit of diffeomorphisms exhibiting
a homoclinic tangency. This motivated the following conjecture~\cite{palis-2005, Bonatti-Survey}, see also~\cite{Crov-ICM}.

\begin{conj*}[Bonatti, Palis]
There exists a dense and open subset $\cU$ of $\Diff^1(M)\setminus \overline \cT$ such that
the diffeomorphism $f\in \cU$ have at most finitely many quasi-attractors (and attractors).
\end{conj*}

More generally, one may consider the chain-recurrence classes of diffeomorphisms~\cite{BC-rec,Crov-habilitation},
which decompose the chain-recurrent dynamics. Bonatti has conjectured~\cite{Bonatti-Survey} that for diffeomorphisms in $\cU$,
the number of chain-recurrence classes is finite.

On surfaces, this conjecture is implied by a stronger result,
proved by Pujals and Sambarino~\cite{Pujals-Sambarino}.
This paper is a step towards this conjecture when $M$ has dimension $3$ and in some regions of $\Diff^1(M)$,
when $M$ has dimension larger than $3$. These results were announced in \cite{Crov-ICM} and \cite{CroPot}.
\medskip

We consider the (open) subset $\PH^1_{c=1}(M)$ of $C^1$-diffeomorphisms $f$ of $M$
which preserve a \emph{partially hyperbolic} decomposition, with a one-dimensional center, i.e.
which preserve a splitting $T M = E^s \oplus E^c \oplus E^u$, $\dim(E^c)=1$, with the property that
for some $\ell >0$ and for every unit vectors  $v^\sigma \in E^\sigma_x$ ($\sigma=s,c,u$) we have
that:
\begin{equation}\label{eq:PH} \|Df_x^\ell v^s \| < \min \{1, \|Df_x^\ell v^c \|\} \leq \max
\{ 1, \|Df_x^\ell v^c \| \} < \|Df_x^\ell v^u \|.
\end{equation}
We will always assume that both $E^s,E^u$ are non-trivial.
Partial hyperbolicity has been playing a central role in the study of differentiable dynamics due to its robustness and how it is related with the absence of homoclinic tangencies (see \cite{Crov-habilitation,CSY}). It also prevents the existence of sinks.

Under some global assumptions it is sometimes possible to show that partially hyperbolic dynamics with one-dimensional center have finiteness and sometimes even uniqueness of quasi-attractors (see e.g. \cite{BG,Pot2}, \cite[Section 6.2]{HP} or \cite[Section 5]{Pot2}). However, it is easy to construct examples of partially hyperbolic diffeomorphisms with infinitely many quasi-attractors (e.g. by perturbing Anosov$\times$Identity on $\TT^3=\TT^2 \times S^1$). Here, we prove that this is a fragile situation:

\begin{teoA}
There exists an open and dense subset $\cO$ of $\PH^1_{c=1}(M)$ such that
every $f\in \cO$ has at most finitely many quasi-attractors.
\end{teoA}

In dimension $3$, we obtain a stronger conclusion:

\begin{teoB}
Let $M$ be a 3-dimensional manifold. There is an open and dense subset $\cU\subset\Diff^1(M)\setminus \overline\cT$
of diffeomorphisms $f$ such that:
\begin{itemize}
\item either $f$ has at most finitely many quasi-attractors,
\item or $f$ is accumulated by diffeomorphisms with infinitely many sinks.
\end{itemize}
\end{teoB}
Another work~\cite{CPS} will address the finiteness of the set of sinks
for diffeomorphisms far from homoclinic tangencies
and will conclude the proof of Bonatti-Palis Conjecture in dimension $3$. We emphasize that this corresponds to a problem of different nature.
\bigskip

More generally we consider invariant compact sets $\Lambda$ which are \emph{partially hyperbolic}, i.e.
which admit a continuous $Df$-invariant splitting $T_\Lambda M = E^s \oplus E^c \oplus E^u$ and $\ell >0$ with the property that for every
unit vectors  $v^\sigma \in E^\sigma_x$ ($\sigma=s,c,u$) the property~\eqref{eq:PH} holds.
Theorem A is a consequence of a more precise result:
\begin{teoC}
There exists a dense G$_\delta$ subset $\cG_1$ of $\Diff^1(M)$ with the following property.
Consider $f_0 \in \cG_1$ and a compact set $U\subset M$ such that $\Lambda = \bigcap_{n\in \ZZ} f_0^n(U)$
is a partially hyperbolic set with one-dimensional center.

Then, for every $f$ $C^1$-close to $f_0$
the set $U$ contains at most finitely many quasi-attractors of $f$.
\end{teoC}

As a consequence we obtain a (weak) version of an unpublished Theorem by
Bonatti-Gan-Li-Yang (\cite{BGLY}).

\begin{corD}
There exists a dense G$_\delta$ subset $\cG_2$ of $\Diff^1(M)$ such
that if $f\in \cG_2$ and $Q$ is a partially hyperbolic
quasi-attractor for $f$ with one dimensional center, then, $Q$ is
not accumulated by other quasi-attractors.
\end{corD}

Such quasi-attractors are called \emph{essential attractors} in
\cite{BGLY} since it follows from their properties that their
basin contains a residual subset in an attracting neighborhood. In
\cite{BGLY} they prove that \emph{every} quasi-attractor for a
$C^1$-generic diffeomorphism $C^1$-far from homoclinic tangencies
is an essential attractor.
\bigskip

\paragraph{\bf Discussion of the techniques.}
The finiteness of the quasi-attractors relies on a geometric property of invariant sets laminated by unstable manifolds:
a \emph{non-joint integrability} between the strong stable and unstable directions, see Figure~\ref{l.NJI}.
Such geometric properties of unstable laminations already appeared for instance in the study of partially hyperbolic attractors~\cite{Pujals,Pujals2,CP}.
When the system is globally partially hyperbolic and volume preserving, a different but related notion -- the accessibility -- plays an important role for proving the ergodicity,
see for instance~\cite{pughshub}.

The main purpose of this work is to break the joint integrability by $C^1$-perturbation.
It was known how to break it for one pair, or even for a dense collection of pairs, of unstable leaves.
We need however to break it for any pair of unstable leaves which intersect a same stable manifold.
This strong form of non joint integrability, much more difficult to obtain, requires a global perturbation.
This perturbative result (Theorem \ref{Thm-PerturbationResult}) is independent of the one-dimensionality of the center direction.  We hope the perturbation result will find applications beyond the ones appearing in this paper (for example, it is used in \cite{ACP} to obtain robust transitivity of $C^1$-generic partially hyperbolic transitive diffeomorphisms with one dimensional center, see also \cite{CroPot}).

\subsection*{Organization of the paper}
In section \ref{Section-Results} we give precise statements of the main technical results. In particular, Theorem \ref{Theorem-MainGenerique} and the remark after provide the geometric property that is satisfied by partially hyperbolic sets saturated by strong unstable manifolds in a $C^1$-open and dense set of diffeomorphisms.  In section \ref{Section-Proofs} we use these statements to give proofs of Theorems A , B, C and Corollary D.

The rest of the paper is devoted to the proofs of Theorems \ref{Theorem-MainGenerique}, \ref{Thm-PerturbationResult} and \ref{Teo-FinitelyManyMinimal}. In section \ref{Section-Preliminaries} some preliminaries are introduced. Section \ref{Section-Perturbations} gives a proof of Theorem \ref{Thm-PerturbationResult}, this section is the technical core of the paper. In Section \ref{s.teogenerique}, using a standard Baire argument, we deduce Theorem \ref{Theorem-MainGenerique} from Theorem \ref{Thm-PerturbationResult}. Finally, in section \ref{Section-Applications} we prove Theorem \ref{Teo-FinitelyManyMinimal}.
\medskip

\noindent{\it Acknowledgments.
We are grateful to D. Yang and J. Zhang for their comments on a first version of the text.}

\section{Technical results}\label{Section-Results}

It is well known that partially hyperbolic sets carry $f$-invariant strong stable and strong unstable laminations $\cW^s$ and $\cW^u$ by $C^1$-leaves tangent to $E^s$ and $E^u$ when intersecting $\Lambda$  (see Section \ref{Section-Preliminaries} for precise definitions and existence theorems).

Given $x\in \Lambda$ we denote by $\cW^\sigma(x)$ the leaf of $\cW^\sigma$ ($\sigma=s,u$) through $x$. For $\eps>0$ we denote by $\cW^\sigma_\eps(x)$ the $\eps$-disk centered at $x$ in $\cW^\sigma(x)$ with the metric $d_\sigma$ given by the Riemannian metric induced in $\cW^\sigma(x)$ from its immersion in $M$.

\begin{teo}\label{Theorem-MainGenerique} There exists a $G_\delta$-dense subset $\cG$ of
$\Diff^1(M)$ such that for every $f \in \cG$ and $\Lambda \en M$ a
compact $f$-invariant partially hyperbolic set which is
$\cW^u$-saturated and for every $r,r',t,\gamma>0$ sufficiently small,
there exists $\delta>0$ with the following property.

If $x,y \in \Lambda$ satisfy $y \in \cW^s(x)$ and $d_s(x,y) \in (r, r')$, then there is $x' \in \cW^u_t (x)$ such that:
\begin{equation}\label{eq:maingenerique} d(\cW^s_\gamma(x'), \cW^u_\gamma(y)) \geq \delta \, .
\end{equation}
\end{teo}

\begin{obs}\label{Remark-OpenGivenScale}
Using continuity of the strong stable and unstable manifolds with respect to the points and the diffeomorphisms one obtains that for every $r,t>0$, (modulo changing slightly the constant $\delta$) the same property holds for $g$ in a $C^1$-small neighborhood of $f$ which depends only on $r$ and $t$. See Lemma \ref{Lemma-RobustProperty} below.\end{obs}

\begin{figure}[ht]
\begin{center}
\includegraphics[scale=0.4]{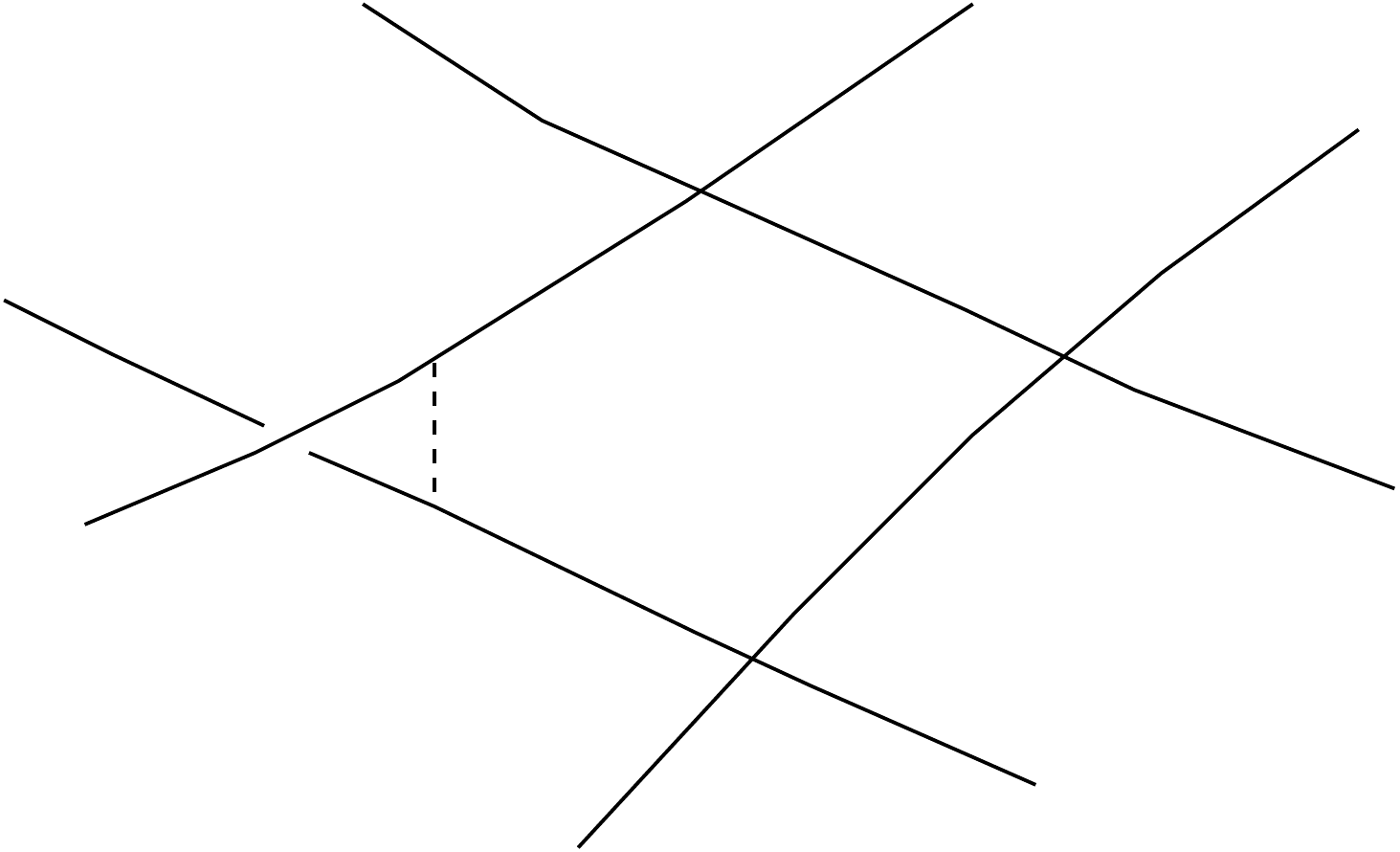}
\begin{picture}(0,0)
\put(-130,26){$x'$}
\put(-140,120){$y$}
\put(-200,5){\small$\cW^{u}(x)$}
\put(-280,50){\small$\cW^{u}(y)$}
\put(-60,99){$x$}
\put(-238,160){$\cW^{s}(x)$}
\put(-290,120){$\cW^{s}(x')$}
\end{picture}
\end{center}
\vspace{-0.5cm}
\caption{The non-joint integrability.\label{l.NJI}}
\end{figure}
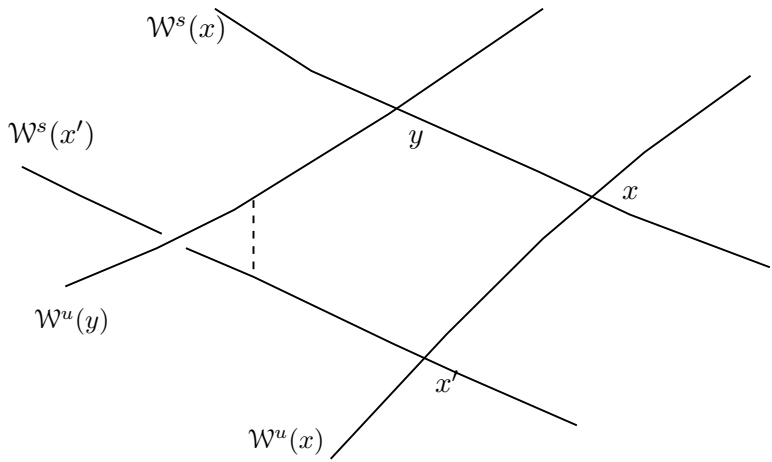

Theorem \ref{Theorem-MainGenerique} works without hypothesis on the dimension of $E^c$. It can be compared to Dolgopyat-Wilkinson's Theorem (\cite{DW}).  Even if one assumes global partial hyperbolicity the result is more general as we need to control non-joint integrability in given $\cW^u$-saturated subsets. Moreover, it gives a quantitative form of non-joint integrability.

The result is a consequence of the following perturbation result and a standard Baire argument.

\begin{teo}\label{Thm-PerturbationResult}
Let $f\colon M\to M$ be a $C^1$-diffeomorphism and $A\subset M$ be compact set such that the maximal invariant set $\Lambda_f$ in $A$ admits a partially hyperbolic splitting of the form $T_{\Lambda_f} M = E^s\oplus E^c \oplus E^u$.

\medskip

Given $\cV$ a $C^1$-neighborhood of $f$ and sufficiently small values of $r,r',t,\gamma>0$, there exists $g \in \cV$ such that the maximal invariant set $\Lambda_g$ in $A$ for $g$ admits a partially hyperbolic splitting into bundles with the same dimensions as the splitting on $\Lambda_f$ and if $x,y \in \Lambda_g$ verify that:
\begin{itemize}
\item $y \in \cW^s_g(x)$ and $d_s(x,y) \in [r, r']$,
\item $\cW^u_{g,t}(x) \en \Lambda_g$,
\end{itemize}
\noindent then there exists $x' \in \cW^u_{g,t}(x)$ such that:
  \begin{equation}\label{eq:perturbationresult} \cW^s_{g,\gamma}(x') \cap \cW^u_{g,\gamma}(y) = \emptyset \, .
  \end{equation}
\end{teo}

\begin{obs}\label{Remark-RegularityVolumePreserving}
The perturbation is made by composing with volume preserving $C^\infty$ diffeomorphisms in certain regions, so, the
perturbation can be made to preserve regularity and volume.

However, the perturbation we make is not enough to preserve a symplectic form (compare with \cite{DW}) since we need to preserve
certain directions which cannot be Lagrangian (see Lemma \ref{Lema-ElementaryPerturbation}).
\end{obs}

The applications we will explore in this paper all assume that
$\dim E^c = 1$. In a certain sense, this result implies that, for
$C^1$-generic diffeomorphisms, a minimal $\cW^u$-saturated set
occupies a ``big space'' in the manifold.

We say that a set $\Gamma$ is a \emph{minimal}
$\cW^u$-\emph{saturated} set if it is minimal for the inclusion among the compact,
$f$-invariant, $\cW^u$-saturated non-empty (partially hyperbolic) sets. Standard arguments imply that every compact $\cW^u$-saturated
partially hyperbolic set $\Lambda$ contains at least one minimal
$\cW^u$-saturated set. Our main results follow from:

\begin{teo}\label{Teo-FinitelyManyMinimal}
There exists a dense G$_\delta$ subset $\cG \en \Diff^1(M)$ such
that for any $f_0 \in \cG$ and any compact $f_0$-invariant partially
hyperbolic set $K$ with $\dim E^c=1$, then
there are neighborhoods $\cU$ of $f_0$ and $U_K$ of $K$
with the following property.

For any $f\in \cU$ and any $f$-invariant compact (partially hyperbolic) set $\Lambda\subset U_K$
which is $\cW^u$-saturated then, $\Lambda$ contains at most finitely many minimal $\cW^u$-saturated
subsets.
\end{teo}

\section{Proofs of Theorems A, B, C and Corollary D}\label{Section-Proofs}

We first prove Theorem C (and hence Theorem A).

\begin{proof}[\sc Proof of Theorem C]
Consider $f_0$ in the G$_\delta$ set $\cG_1:=\cG$ given by Theorem~\ref{Teo-FinitelyManyMinimal}
and $\Lambda= \bigcap_{n\in \ZZ} f_0^n(U)$ a partially hyperbolic set in a compact set $U$.
Then, for any $f$ that is $C^1$-close to $f_0$, the maximal invariant set $\Lambda_f$ in $U$
is still partially hyperbolic.
Since every quasi-attractor contains at least one minimal $\cW^u$-saturated set and different quasi-attractors are disjoint, Theorem C is a direct consequence of Theorem \ref{Teo-FinitelyManyMinimal}.
\end{proof}

To prove Theorem B we need the following result.

\begin{teo}[\cite{CSY}]\label{teo-CPSY}
Let $f$ be a $C^1$-generic diffeomorphism far from homoclinic tangencies.
Then, every chain recurrence class $C$
has a dominated splitting $T_CM=E^s\oplus E^c_1\oplus\dots\oplus E^c_k\oplus E^u$
where $E^s$, $E^u$ are uniformly expanded and contracted and $E^c_1,\dots,E^c_k$ are one-dimensional.

Moreover:
\begin{itemize}
\item If $C$ does not contain any periodic point, $E^s$ and $E^u$ are non-trivial.
\item If $C$ contains some periodic point, then for each $i=1,\dots,k$,
it contains periodic orbits whose Lyapunov exponent
along $E^c_i$ is arbitrarily close to $0$.
\end{itemize}
\end{teo}

We emphasize that in this paper partial hyperbolicity requires the extremal bundles to be non-trivial.

\begin{obs}\label{r.CPSY}
In the second case of Theorem \ref{teo-CPSY},
the class $C$ is limit for the Hausdorff topology of hyperbolic periodic orbits
whose stable dimension is $\dim(E^s)$ and other ones with stable dimension is $\dim(E^s\oplus E^c)$.
\end{obs}

\begin{proof}[\sc Proof of Theorem B]
We first introduce the dense G$_\delta$ subset $\cG$ of the space $\diff 1\setminus \overline{\cT}$
of diffeomorphisms $f$ satisfying:
\begin{itemize}
\item $f$ is a continuity point of the map $\Gamma$
which associate to a diffeomorphism $g$,
the numbers
of its quasi-attractors.
\item $f$ and $f^{-1}$ satisfy Theorems \ref{teo-CPSY} and \ref{Teo-FinitelyManyMinimal}.
\end{itemize}
Indeed the map $\Gamma$
varies semi-continuously with the diffeomorphism, hence
its continuity points is a residual set.

We now prove the conclusion of Theorem B for diffeomorphisms $f$ in $\cG$.
More precisely:
\smallskip

\paragraph{\it Case 1.} If $f$ has finitely many quasi-attractors,
 then the same holds for diffeomorphisms $C^1$-close, since $\cG$ is made by continuity points of $\Gamma$. (This is a classical argument, see e.g. \cite{Crov-habilitation} for similar arguments.)
\smallskip

\paragraph{\it Case 2.} If $f$ has infinitely many quasi-attractors, then it must have infinitely many sinks.
Otherwise let us assume by contradiction that $f$ has infinitely many quasi-attractors $Q_n$ that are non-trivial, i.e. not sinks.
One may assume that $(Q_n)$ converges for the Hausdorff topology towards an invariant compact set $K$,
contained in a chain-recurrence class. The set $K$
is not a sink nor a source and can not be uniformly hyperbolic.
By Theorem \ref{teo-CPSY}, it admits a dominated splitting $T_K M=E^s\oplus E^c_1\oplus\dots\oplus E^c_k\oplus E^u$ where $E^c_i$ are one-dimensional.
As there are finitely many sinks, Remark \ref{r.CPSY} implies that the bundles $E^s$ and $E^u$ are non-trivial. Since $M$ is 3-dimensional and $K$ is not hyperbolic, it follows that $K$ must be partially hyperbolic:
it has a dominated splitting $E^s\oplus E^c\oplus E^u$,
where $E^s,E^c, E^u$ are non-trivial (and one-dimensional). The same holds for the maximal invariant set $\Lambda$
in a neighborhood of $K$.
The $Q_n$ for $n$ large are contained in $\Lambda$ and are $\cW^u$ saturated since they are quasi-attractors.
This contradicts Theorem \ref{Teo-FinitelyManyMinimal}.
\smallskip

Let $\cO$ be the set of diffeomorphisms $f\in \diff 1\setminus \overline{\cT}$ such that any diffeomorphism
$C^1$ close has finitely many quasi-attractors.
We set
$$\cU=\cO\cup (\diff 1\setminus \overline{\cT\cup \cO}).$$
It is dense and open in $\diff 1\setminus \overline{\cT}$.

Consider $f$ in $\cU$. If $f\in \cO$, the first conclusion of Theorem B holds.
Otherwise, $f$ is limit of diffeomorphisms $g_n\in \cG\setminus \cO$.
Since $g_n\not\in \cO$, the case 2 above holds and so each $g_n$ has infinitely many sinks.
This proves TheoremB.
\end{proof}

\begin{proof}[\sc Proof of Corollary D] This follows directly by applying Theorem~C to the maximal invariant set in a
neighborhood of $Q$.
\end{proof}

\section{Preliminaries}\label{Section-Preliminaries}

\subsection{Partial hyperbolicity and cone-fields}\label{ss.cones}

Let $M$ be a Riemannian manifold. Let $E_x \en T_xM$ be a subbundle. We define the $\eps$-\emph{cone} around $E_x$ to be

$$ \cE^E_\eps (x) := \{ v \in T_xM \ : \ v=v^E + v^{E^\perp} \ : \ \|v^{E^\perp} \| < \eps \|v^E \| \} \cup \{0\} $$

\noindent where $E^\perp$ is the orthogonal subbundle of $E$ and for every $v \in T_xM$ we denote by $v= v^E + v^{E^\perp}$  the unique decomposition of $v$ in vectors of $E$ and $E^\perp$. We will denote by $\overline{\cE}^E_\eps(x)$ the closure of $\cE^E_\eps(x)$ in $T_xM$.

Given a continuous bundle $E \en T_K M$ in a compact set $K \en M$ one can always define a continuous extension of $E$ to a neighborhood $U$ of $K$ (see \cite[Proposition 2.7]{CroPot}).

A \emph{continuous cone-field} $\cE^E$ in $U$ of width $\eps: U \to \RR_{>0} \in C^0(U, \RR)$ around $E$ is given by assigning to each $x \in U$ the cone $\cE^E_{\eps(x)}(x)$. We say the cone-field has \emph{dimension} $\dim E$ and \emph{width} $\eps$. The \emph{angle} of the cone-field will be $\sup_{x \in U} \{ \eps(x) \}$.

Given $f: M \to M$ a $C^1$-diffeomorphism, we say that the cone-field $\cE^E$ defined in $U$ is \emph{strictly} $Df$-\emph{invariant} if for every $x\in U \cap f^{-1}(U)$ we have that

$$D_xf (\overline{\cE}^E_{\eps(x)}(x)) \en \cE^E_{\eps(f(x))}(f(x))$$

Let $f: M \to M$ be a $C^1$-diffeomorphism and $\Lambda \en M$ a compact $f$-invariant partially hyperbolic set.

Once $\Lambda$ is fixed, we can choose a continuous Riemannian metric given by \cite{Gourmelon-Addapted} such that:

\begin{itemize}
\item All bundles of the partially hyperbolic splitting are orthogonal in $\Lambda$.
\item Vectors in $E^s$ are expanded by $Df^{-1}$ and vectors in $E^{u}$ are expanded by $Df$ on $\Lambda$.
\item For any constant $\varepsilon>0$, the cones $\cE^u_\varepsilon$ and $\cE^{cu}_\varepsilon$ on $\Lambda$
of constant width $\varepsilon$ around $E^u$ and $E^{c}\oplus E^u$ respectively are strictly $Df$-invariant.
\item For any constant $\varepsilon>0$, the cones $\cE^s_\varepsilon$ and $\cE^{cs}_\varepsilon$ on $\Lambda$
of constant width $\varepsilon$ around $E^s$ and $E^{s}\oplus E^c$ respectively are strictly $Df^{-1}$-invariant.
\end{itemize}

We fix the width of the cones $\varepsilon_0>0$ once and for all. By continuity,
we fix a small compact neighborhood $U_1$ of $\Lambda$ such that:

\begin{itemize}

\item There exist continuous and strictly $Df$-invariant cone-fields $\cE^u, \cE^{cu}$ of constant width which are defined in $U_1$,
extending $\cE^u_{\varepsilon_0}$ and $\cE^{cu}_{\varepsilon_0}$.
\item There exist continuous and strictly $Df^{-1}$-invariant cone-fields $\cE^s, \cE^{cs}$ of constant width defined in $U_1$
extending $\cE^s_{\varepsilon_0}$ and $\cE^{cs}_{\varepsilon_0}$.
\item Vectors in $\cE^s$ are expanded by $Df^{-1}$ and vectors in $\cE^{u}$ are expanded by $Df$.
\end{itemize}

We will choose a smooth Riemannian metric close to the continuous one above where all properties still hold except that orthogonality of the bundles is slightly perturbed. This Riemannian metric will remain fixed.

We will not prove this proposition as it is standard, we refer the reader to \cite[Section 2.2]{CroPot} for a proof.

\begin{prop}\label{Proposition-Cones}
Let $f: M \to M$ be $C^1$-diffeomorphism and $\Lambda$ a partially hyperbolic set. Then, there exists a neighborhood $U$ of $\Lambda$ and a neighborhood $\cU$ of $f$ such that if $\cE^s, \cE^u, \cE^{cs}, \cE^{cu}$ are the continuous cone-fields for $f$ defined as above then the following holds:

\begin{itemize}
\item[(i)] For every $g\in \cU$, the cone-fields $\cE^{u},\cE^{cu}$ are strictly $Dg$-invariant and $\cE^{s}, \cE^{cs}$ are strictly $Dg^{-1}$-invariant in $U$.

\item[(ii)] For every $g \in \cU$, vectors in $\cE^u$ are expanded by $Dg$ and vectors in $\cE^{s}$ are expanded by $Dg^{-1}$.

\item[(iii)] For every $a >0$ there exists $N_a>0$ such that if $g\in \cU$ and $x \in \overline U \cap \ldots \cap g^{-N_a}(\overline U)$ then $D_xg^{N_a}(\cE^{\sigma}(x))$ is contained in an cone of width $a$ of $T_{g^N(x)}M$ of the same dimension ($\sigma=u,cu$).

\item[(iv)] For every $a>0$ there exists $N_a>0$ such that if $g\in \cU$ and $x \in \overline U \cap \ldots \cap g^{N_a}(\overline U)$ then $D_xg^{-N_a} (\cE^{\sigma}(x))$ is contained in an cone of width $a$ of $T_{g^{-N_a}(x)}M$ of the same dimension ($\sigma=s,cs$).
\end{itemize}
\end{prop}

\begin{obs}\label{Remark-PersistencePH} As a consequence we know that the maximal invariant subset of $g \in \cU$ in $U$ is also partially  hyperbolic and the invariant bundles $E^\sigma$ for $g$ over the maximal invariant set are contained in the cones $\cE^\sigma$ ($\sigma=s,cs,cu,u$). In particular, the dimensions of the bundles remain unchanged and the bundles themselves vary continuously with respect to the points and the diffeomorphism.
\end{obs}

\subsection{Stable-Unstable Laminations}

The following classical result (see for example \cite{CroPot}, \cite{HPS} or \cite[Appendix B]{BDV}) will be important in our study.

We let $f: M \to M$ be a $C^1$-diffeomorphism and $\Lambda$ a compact partially hyperbolic set for $f$. We fix the Riemannian metric chosen in the previous subsection and compact neighborhoods $U_1 \en U$ of $\Lambda$ and $\cU$ of $f$ such that the cone-fields in Proposition \ref{Proposition-Cones} are well defined in $U$ and for every $g \in \cU$ satisfy the conditions in Proposition \ref{Proposition-Cones}. We define:

$$ \Lambda_g = \bigcap_{n\in \ZZ} g^n( U_1) $$

\noindent Since $\Lambda_g$ is partially hyperbolic for $g$ we denote its bundles by $E^\sigma_g$ ($\sigma=c,s,u$).

Here $\DD^d$ denotes the Euclidean closed disk of radius $1$ of dimension $d$.

\begin{teo}[Stable Manifold]\label{Teo-StableManifold}
Given $g\in \cU$ there is a function $\cD^s_g : E^s_g|_{\Lambda_g} \to U_1$ such that
if we denote  $\cD^s_{g,x} = \Image \cD^s_g|_{E^s_g(x)}$ then:

\begin{itemize}
\item if $0_x$ is the zero vector in $E^s_{g,x}$, then $\cD^s_g(0_x)= x$ and $T_x \cD^s_{g,x}= E^s_{g,x}$.

\item \emph{(Trapping property)} $g(\overline{\cD^s_{g,x}}) \en \cD^s_{g,g(x)}$.

\item \emph{(Tangency to the bundle)} For every $x \in \Lambda_g$ and $z \in \cD^s_{g,x}$ one has that $T_z \cD^s_{g,x} \en \cE^s(z)$.

\item \emph{(Convergence)} For every $y \in \cD^s_{g,x}$ one has that $d(g^n(y),g^n(x)) \to 0$ exponentially as $n \to +\infty$.

\item \emph{(Coherence)} For every $x,y \in \Lambda_g$ one has that $\cD^s_{g,x} \cap \cD^s_{g,y}$ is relatively open in $\cD^s_{g,x}$ and $\cD^s_{g,y}$.

\end{itemize}

Moreover, the map $g \mapsto \cD^s_g$ is continuous in the sense that if $g_n \to g$ and $x_n \in \Lambda_{g_n}$ verifies that $x_n \to x \in \Lambda_g$, then, the maps $\cD^s_{g_n}|_{E^s_{g_n,x_n}}$ converge\footnote{Notice that by Remark \ref{Remark-PersistencePH} the bundles $E^s_{g_n,x_n}$ converge to $E^s_{g,x}$ so that this convergence makes sense.} in the $C^1$-topology to $\cD^s_{g}|_{E^s_{g,x}}$.
\end{teo}

The same property holds for $f^{-1}$ where the roles of $E^s$ and $E^u$ are interchanged giving rise to a family $\cD^u_{g,x}$.

With these objects, we can introduce the \emph{stable and unstable laminations}. For $g \in \cU$ and $x \in \Lambda_g$ we define:

$$ \cW^s_g(x) = \bigcup_{n \geq 0} g^{-n}(\cD^s_{g,g^n(x)}) $$

We will call $\cW^s_g(x)$ the \emph{strong stable manifold} of $x$ for $g$. We define the strong unstable lamination $\cW^u_g$ similarly (its leafs are called \emph{strong unstable manifolds}) using $g^{-1}$.

When we work with $f$ and the original set $\Lambda$ we will not make reference to $f$. Notice however that in principle $\Lambda\neq \Lambda_f$.

We consider in each leaf $\cW^\sigma_{g}$ ($\sigma=s,u$) of the lamination the metric induced by the Riemannian metric chosen in subsection \ref{ss.cones}.We denote the metric inside the leafs of the laminations $\cW^\sigma_g$ by $d_\sigma$. When we write $d_\sigma(x,y)$ it is implicit that $x$ and $y$ lie in the same leaf.

For $\eps>0$, we denote by $\cW^\sigma_{g,\eps}(x)$ the set of points $y \in \cW^\sigma_g(x)$ such that $d_\sigma(x,y)\leq \eps$.

\begin{obs}\label{Remark-ContractionStable} There exists $\lambda \in (0,1)$ and $\nu_0 >0$ such that for every $0<\eps< \nu_0$ and for every $g\in \cU$, $x \in \Lambda_g$ it holds that:
$$ g (\cW^s_{g,\eps}(x)) \en \cW^s_{g,\lambda \eps}(g(x)) \ ; $$
$$ g^{-1}(\cW^u_{g,\eps}(x)) \en \cW^s_{g, \lambda \eps} (g^{-1}(x)) \, .$$
\end{obs}

In this sense, the manifolds $\cW^\sigma_{g,\eps}(x)$ play a similar role as the ones $\cD^\sigma_{g,x}$ in the statement of Theorem \ref{Teo-StableManifold}. We will use $\cW^\sigma_{g,\eps}(x)$ in this paper.

\subsection{Local product structure}

The continuity of the bundles allows to, at a sufficiently small scale, control the geometry of manifolds tangent to them.
Since $E^{cu}$ and $E^s$ are almost orthogonal, this gives:

\begin{prop}\label{Prop-LocalProductStructure}
Consider a $C^1$-diffeomorphism $f: M \to M$, a partially hyperbolic set  $\Lambda$  and open sets $\cU$, $U_1$ and $U$ defined in the previous sections.  There exists a constant $\gamma_0>0$ such that for all $\gamma<\gamma_0$ and
\begin{itemize}
\item for every $g \in \cU$
\item for every $x,y \in \Lambda_g$ such that $d(x,y)<\gamma$ and
\item for every $D$  disk tangent to $\cE^{cu}$ of diameter $2\gamma$ and centered in $y$,
\end{itemize}
\noindent we have that there exists a unique point $z$ in the intersection of $\cW^s_{g,\gamma}(x)$ and $D$.
Moreover, $d(x,z)$
and $d(y,z)$ are smaller than $2d(x,y)$.
\end{prop}
 We will assume from now on that $\gamma$ is smaller than $\nu_0$ of Remark \ref{Remark-ContractionStable}.

\section{Perturbation results}\label{Section-Perturbations}

The purpose of this section is to prove Theorem \ref{Thm-PerturbationResult}.

As before, we will consider $f: M \to M$ a $C^1$-diffeomorphism and $\Lambda$ a compact $f$-invariant set which admits a partially hyperbolic
splitting of the form $T_\Lambda M = E^s\oplus E^c \oplus E^u$.

We will fix $s= \dim E^s$, $u=\dim E^u$ and $c= \dim E^c$. We have that $d= \dim M = s+c+u$.

\subsection{Initial constructions}\label{ss.initial}
Let us start by extending continuously the bundles of the dominated splitting as a (non-invariant) splitting $TM=E_1\oplus E_2\oplus E_3$ to a neighborhood $U$ of $\Lambda_f$.
This allows to consider the constant width cones $\cE^\sigma_\alpha$ for any $\alpha>0$ as in Section~\ref{ss.cones}.

Up to reduce $U$, this also defines continuous cone fields $\cE^{\sigma}$, $\sigma\in \{s,cs\}$ (resp. $\sigma\in \{u,cu\}$) which are invariant by $Df^{-1}$ (resp. $Df$),
satisfy the properties stated in Proposition \ref{Proposition-Cones} 
and are contained in the cones $\cE^{\sigma}_{1/10}$.

We recover the constants $C,\gamma$ and open sets $\cU$, $U$ and $U_1$ of Proposition \ref{Prop-LocalProductStructure}.

We introduce $\Delta= \max_{x \in U  ,  g\in \cU} \{ \|D_x g \|, \|D_x g^{-1} \| \}$.

\subsection{Strategy and plan of the proof}

We start by giving the strategy of the proof and a plan on how we will implement it.

The proof has three stages:

\begin{itemize}
\item Construction of the local perturbation which breaks locally the joint integrability. For this, it is important that the perturbation is made in a region which is disjoint from many iterates so that one can alter the position of one bundle in a way that the other bundle is almost untouched.

\item Construction of wandering slices disjoint from several iterates.

\item Construction of sections of the unstable lamination on which perturbations will be placed. In order to break all possible joint integrability these sections have to cover all the local strong unstable manifolds.
\end{itemize}

Local $C^1$ perturbations breaking joint integrability already appeared before (for instance in \cite{DW}). The main novelty here is the last stage where we achieve the control of all local unstable leafs.

We now give the plan of the proof to help the reader transit this section.

First, in subsection \ref{ss.cubes} we give some fixed coordinates in a neighborhood $U$ of $\Lambda_f$ which respect the bundles of the partial hyperbolicity and introduce in subsection \ref{ss.elementary} the elementary perturbation we will later place in several parts of the manifold.

In subsection \ref{ss:perturbation1} we define the perturbation of $f$ that will give the proof of Theorem \ref{Thm-PerturbationResult} but do not specify in which places the elementary perturbations will take place. This section already shows the type of perturbation which one will make, but one needs to wait to subsection \ref{ss:perturbation2} to see where the elementary perturbations will be placed.

In subsections \ref{ss:controlbundles} and \ref{ss:breakjoint} we control the effects of the perturbations and show that for points whose unstable manifolds are captured by the perturbation, the joint integrability is broken.

In subsections \ref{ss:coverings}, \ref{ss:wandering}, \ref{ss:section} and \ref{ss:choicerect} we construct the wandering slices and the section of the unstable manifold so that the perturbation has the desired effect.

Since the proof is involved and the choice of constants is quite subtle, we sum up all the choices of constants made in subsection \ref{ss:perturbation2} to show clearly that the choices are consistent. Finally, in subsection \ref{ss:check} we check that the perturbation we made proves Theorem \ref{Thm-PerturbationResult}.

\subsection{Cubes adapted to the dominated splitting.}\label{ss.cubes}
Given a sufficiently small $\rho>0$, we define the \emph{$\rho$-cube $C_\rho^x$} at $x\in U_0$ to be the image of the cube $[-\frac{\rho}{2},\frac{\rho}{2}]^{s} \times [-\frac{\rho}{2},\frac{\rho}{2}]^{c} \times [-\frac{\rho}{2},\frac{\rho}{2}]^{u}$ in the coordinates $T_xM =E_{1}(x) \oplus E_{2}(x) \oplus E_3(x)$ by the exponential map $\exp_x : T_x M \to M$.

One can take a coordinate map $\psi^x_\rho: C^x_\rho \to [0,1]^d$ which consists in composing the inverse of the exponential map with an affine map from $T_xM$ to $\RR^d$ which is an homothety of ratio $\rho^{-1}$ on each space $E_i(x)$ and sends the splitting $E_1(x)\oplus E_2(x)\oplus E_3(x)$ to the canonical splitting of $\RR^d = \RR^{s} \oplus \RR^{c} \oplus \RR^{u}$.

\smallskip

A \emph{slice of width $\eta$} inside a $\rho$-cube $C^x_\rho$ with coordinates $\psi^x_\rho$ is a subset $T \en C^x_\rho$ of the form
$$T=(\psi^{x}_\rho)^{-1}(([0,1]^{s+c} \times [0,\eta]^u)+ v),$$
\noindent where $v \in \{0\}^{s+c} \times [0, 1 -\eta]^u$. We also define a \emph{rectangle of width $\eta$} inside the $\rho$-cube to be a subset $R \en C^x_\rho$ of the form
$$R=(\psi^{x}_\rho)^{-1}(([0,\eta]^s \times [0,1]^{c} \times [0,\eta]^{u})+ w),$$
\noindent with $w \in [0,1-\eta]^s \times \{0\}^{c} \times [0, 1-\eta]^u$.

\begin{figure}[ht]
\begin{center}
\includegraphics[scale=0.4]{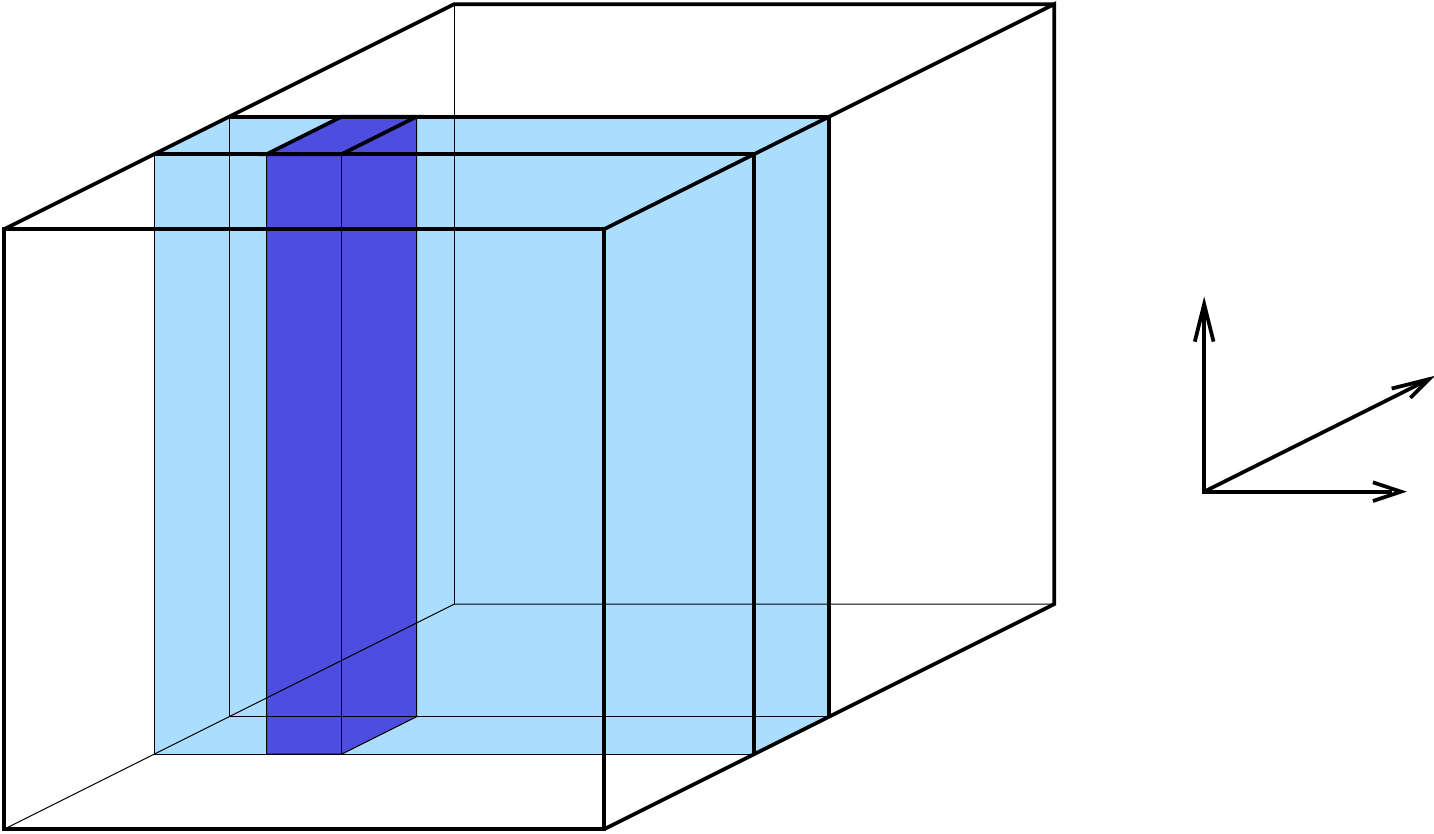}
\begin{picture}(0,0)
\put(-211,91){$R$}
\put(-130,62){$T$}
\put(-105,100){$C_\rho^x$}
\put(-42,100){$c$}
\put(-10,68){$s$}
\put(-10,90){$u$}
\end{picture}
\end{center}\end{figure}

We have chosen the cones $\cE^{u}$ thin enough so that any disc in a cube $C_\rho^x$ which intersects $C_{\rho/3}^x$ and is tangent to $\cE^u$ is contained in the region $(\psi^x_\rho)^{-1}([\frac 1 4,\frac 3 4]^{s+c}\times [0,1]^u)$.


\subsection{An elementary perturbation}\label{ss.elementary}
The following lemma will allow to perform local perturbations in $M$
and to break the joint integrability near a point.

\begin{lema}\label{Lema-ElementaryPerturbation} Given $\kappa >0$ and $\eta \in (0,1)$ there exists a $C^\infty$ diffeomorphism
$$ h^{\eta}_\kappa : [0, \eta]^s \times [0,1]^c \times [0, \eta]^u \to  [0, \eta]^s \times [0,1]^c \times [0, \eta]^u $$
which is the identity in a neighbourhood of the boundary, which is $\kappa$-$C^1$-close to the identity and which has the following properties:
\begin{itemize}
\item[(i)] For every disk $\cD^u:=\{x^s\} \times \{x^c \} \times [0, \eta]^u$ with $(x^s,x^c)\in [\frac 1 4  \eta,\frac 3 4 \eta]^s\times [1/4,3/4]^c$, the image $h^{ \eta}_\kappa(\cD^u)$ contains two points whose second coordinates differ by more than $\frac{\kappa \eta}{10}$ and whose third coordinates belongs to $[\frac 1 4 \eta,\frac 3 4 \eta]^u$.

\item[(ii)] The map does not change the first coordinate, i.e. $h_\kappa^\eta$ has the form
\begin{equation}\label{eq:elementary} h_\kappa^{\eta}(x^s,x^c,x^u) = (x^s, (h_\kappa^{\eta})_{x^s}(x^c,x^u) ).
\end{equation}
\end{itemize}
\end{lema}

\begin{obs} As it can be seen in the proof below, the map $h_{\kappa}^{\eta}$ preserves all the (one-dimensional) coordinate axes, but two - one in the second bundle and one in the third. In particular, one can obtain the points differing in any direction of the central bundle. This gives enough freedom along the second bundle and one can use this result to obtain a \emph{local accessibility} (compare with \cite{DW}).

Also, one can easily adapt the construction in order to preserve a given volume form instead of the canonical one up to adjusting the construction (c.f. Remark \ref{Remark-RegularityVolumePreserving}).
\end{obs}

\begin{proof}[\sc Proof] In coordinates $(x_1, \ldots, x_{d})$ we choose any given $x_i$ with $s+1 \leq i \leq s+c$ and $x_j$ with $s+c+1 \leq j \leq d$.

We define a hamiltonian (i.e. a smooth function) $H:  [0,1] \times [0,\eta] \to \RR$ such that:

\begin{itemize}
\item it is constant in a neighborhood of the boundary,
\item it is equal to $\kappa x_i$ for $(x_i,x_j) \in [\frac{1}{4}, \frac{3}{4}]\times [\frac 1 4  \eta,\frac 3 4 \eta]$,
\item its derivative is everywhere smaller than $5\kappa.$
\end{itemize}

Choose a smooth bump function $\phi: [0, \eta]^{s} \times [0,1]^{c-1} \times [0, \eta]^{u-1} \to [0,1]$ which equals $0$ in a neighborhood of the bundary, equals $1$ for points in $ [\frac 1 4  \eta,\frac 3 4 \eta]^s \times [\frac{1}{4},\frac{3}{4}]^{c-1} \times  [\frac 1 4  \eta,\frac 3 4 \eta]^{u-1}$ and has derivative everywhere bounded by $5$. We denote the coordinates in $[0, \eta]^{s} \times [0,1]^{c-1} \times [0, \eta]^{u-1}$ to be:

$$ x^s := (x_1, \ldots, x_s) \ ; $$  $$ \hat x^c := (x_{s+1}, \ldots, x_{i-1}, x_{i+1} \ldots, x_{s+c}) \ ;  $$ $$ \hat x^u := (x_{s+c+1},\ldots, x_{j-1}, x_{j+1}, \ldots, x_d). $$

We consider the diffeomorphism $h_{\kappa}^\eta$ which, for each $\mathrm{x}:=(x^s,\hat x^c, \hat x^u)$ preserves the rectangle
$$R_{\mathrm{x}}:=\{(x_1, \ldots, x_{i-1})\} \times [0,1] \times \{(x_{i+1}, \ldots, x_{j-1})\} \times [0,\eta] \times \{(x_{j+1}, \ldots, x_d)\}$$
\noindent It coincides on $R_x$ with the time one map of the hamiltonian flow given by the hamiltonian $\hat H_{\mathrm{x}} : R_{\mathrm{x}} \to \RR$ which maps $(x_i,x_j) \mapsto \phi (x^s, \hat x^c, \hat x^u)  H(x_i,x_j)$. That is, $h_{\kappa}^\eta(x_1,\ldots, x_d)$ in $R_{\mathrm{x}}$ coincides with the solution of the equation:

\[
\left\{ \begin{array}{c}
\dot{x_i} \ =  \ \frac{\partial \hat H_{\mathrm{x}}}{\partial x_j } \\
\dot{x_j}  \ = \ - \frac{\partial \hat H_{\mathrm{x}}}{\partial x_i}
\end{array} \right.
\]
This has the desired properties.
\end{proof}


\subsection{Perturbation in the manifold}\label{ss:perturbation1}
The diffeomorphism $g$ in Theorem~\ref{Thm-PerturbationResult} from $f$ will be obtained as follows:

\begin{itemize}
\item one chooses $\rho,\kappa,\eta>0$ sufficiently small and a large integer $N\geq 1$ whose values will depend on conditions appearing later,

\item one chooses finitely many $\rho$-cubes (not-necessarily disjoint) $C_\rho^{i}$ (with a chart $\psi^i_\rho : C^i_\rho \to [0,1]^d$ as above),

\item inside each cube $C_\rho^{i}$  one chooses some finite collections of slices $T_{i,j}$ of the form
$$T_{i,j}=(\psi^{i}_\rho)^{-1}([0,1]^{s+ c} \times [0,\eta]^u+u_{i,j}),$$ where $u_{i,j}\in \{0\}^{s+c}\times [0,1-\eta]^u$,

\item inside each slice $T_\rho^{i,j}$  one chooses some finite collections of disjoint rectangles $R_{i,j,k}$ of the form
$$R_{i,j,k}=(\psi^{i}_\rho)^{-1}([0, \eta]^s \times [0,1]^c \times [0, \eta]^u+u_{i,j}+v_{i,j,k}),$$ where $v_{i,j,k}\in [0,1-\eta]^s\times\{0\}^{c+u}$,

\item we will ensure the interiors of $f^{-\ell}(T_{i,j})$, for $0\leq \ell\leq N$ and for every $(i,j)$, to be pairwise disjoint; in particular the interiors of
$f^{-\ell}(R_{i,j,k})$, for $0\leq \ell\leq N$ and for every $(i,j,k)$, are pairwise disjoint,

\item we define $g$ to coincide with $f$ outside the union $f^{-1}(\bigcup R_{i,j,k})$ and with
$$(\psi^{i}_\rho)^{-1}\circ \tau_{i,j,k}\circ  h_\kappa^{\eta}\circ \tau_{i,j,k}^{-1}\circ \psi^{i}_\rho\circ f$$
inside each preimage $f^{-1}(R_{i,j,k})$, where $\tau_{i,j,k}$ denotes the translation of $\RR^d$ by $u_{i,j}+v_{i,j,k}$.
\end{itemize}

We have the following property:

\begin{lema}\label{Lema-ManyElementaryPerturbations} Given $\cV$ a $C^1$-neighborhood of $f$ there exists $\rho_0>0$, $\kappa_0>0$ such that if $\rho<\rho_0$ and $\kappa<\kappa_0$, then the resulting diffeomorphism $g$ lies in $\cV$.
\end{lema}

\begin{proof}[\sc Proof] The fact that the perturbation is of $C^0$-size smaller than $\rho$ is direct from the fact that the support of the perturbation is contained in disjoint balls of radius smaller than $\rho$.

To control the $C^1$-size of the perturbation, we make the following remark: There exists a number $\Theta>0$ independent of $\rho$ such that if $h: [0,1]^d \to [0,1]^d$ is a diffeomorphism which is the identity in the boundary then, for any $x$, the $C^1$-size of $\tilde h: C^x_\rho \to C^x_\rho$ given by $\tilde h  = (\psi_\rho^x)^{-1} \circ h \circ \psi^x_\rho$ is at most $\Theta \|h \|_{C^1}$. From the form of the elementary perturbation (Lemma \ref{Lema-ElementaryPerturbation}) we get that if we choose $\kappa$ small enough, the $C^1$-size will be small.

Notice that making $C^1$-small perturbations with disjoint support, one obtains a small $C^1$-perturbation (see for instance \cite[Section 2.9]{Crov-habilitation}).
\end{proof}

This Lemma provides the only restriction on $\kappa$  but we will need to consider more constraints on $\rho$. From now on, $g$ will denote the diffeomorphism obtained as described above.

For each $i,j,k$, we denote by $C_{\rho/3}^{i}$ the $\rho/3$-cube having the same center as $C^i_\rho$ and by $\frac 1 2 R_{i,j,k}$ the sub-rectangle
$$\frac 1 2 R_{i,j,k}=(\psi^{i}_\rho)^{-1}\left(\bigg(\bigg[\frac { \eta} {4} ,\frac {3 \eta} 4\bigg]^s \times [0,1]^c \times \bigg[\frac { \eta} {4} ,\frac {3 \eta} 4\bigg]^u\bigg)+u_{i,j}+v_{i,j,k}\right).$$


\subsection{Control of the new dominated splitting}\label{ss:controlbundles}
We introduce two more constants $\theta,\alpha$, that will be defined later.
They will restrict the choice of $\rho$ to be smaller than some constant
$\rho_1$.

\begin{prop}\label{Prop-NarrowingCones} For any $\theta >0$ there exist $N>0$ and a neighbourhood $U_\theta$ of $\Lambda_f$ such that if $g^{-N+1}(x),\dots,g^{-1}(x),x$ belong to $U_\theta$ and do not intersect the rectangles $R_{i,j,k}$, then $Dg^N(\cE^u(g^{-N}(x)))\subset \cE^u_\theta(x)$ and $Dg^{-N}(\cE^s(x))\subset \cE^s_\theta(g^{-N}(x))$.
\end{prop}
\begin{proof}[\sc Proof] Proposition \ref{Proposition-Cones} (iii) and (iv) applied to $f$ and $\theta/2$ gives and integer $N$ such that $Df^N(\cE^u(x)) \subset \cE^u_{\theta/2}(f^N(x))$
{ and $Df^{-N}(\cE^s(x))\subset \cE^s_\theta(f^{-N}(x))$}
in a neighbourhood $U_\theta$ of $\Lambda_f$.

If the segment of orbit $g^{-N+1}(x),\dots,g^{-1}(x),x$ does not intersect the rectangles $R_{i,j,k}$, it coincides with a segment of orbit of $f$
and the proposition follows.
\end{proof}

\begin{prop}\label{Prop-NarrowingCones2} For any $\alpha >0$ there exists $\rho_1>0$ and a neighborhood $U_\alpha$ of $\Lambda_f$ such that if $\rho<\rho_1$, then
\begin{itemize}
\item $\cE^{cu}_\alpha(x)$ is $Dg$-invariant in $U_\alpha$,
\item $\Lambda_g\subset U_\alpha$.
\end{itemize}
\end{prop}

\begin{proof}[\sc Proof]
By Section~\ref{ss.cones}, there exists $\alpha'\in (0,\alpha)$ such that the cones
$\cE^{cu}_{\alpha}$ are $Df$-invariant in $\Lambda_f$ and mapped inside $\cE^{cu}_{\alpha'}$.
By Section~\ref{ss.initial} and continuity, there exists a compact neighborhood $U_\alpha$ of $\Lambda_f$ such that
these properties extend on $U_\alpha$.

On each $\rho$-cube at a point $z$,
one can consider cone fields $\hat \cE^{cu}_\varepsilon$ of constant width $\varepsilon$,
defined around the coordinates directions (so that at the center $z$ of the $\rho$-cube,
$\hat \cE^{cu}_\varepsilon=\cE^{cu}_\varepsilon$).
If $\rho_2>0$ is small enough and $\rho<\rho_2$
for any $x\in U_\alpha$ in a $\rho$-cube $C_\rho^z$, the distance $d(E_i(x),E_i(z))$ is as small as desired
for $i\in \{1,2,3\}$. In this way, there exists $\alpha''\in (\alpha',\alpha)$
such that
$$\cE^{cu}_{\alpha'}(x)\subset \hat \cE^{cu}_{\alpha''}(x) \subset \cE^{cu}_{\alpha}(x).$$

The elementary perturbation (recall \eqref{eq:elementary}) in a $\rho$-cube preserves the cones $\hat \cE^{cu}_\varepsilon$.
We have proved the proposition for $\cE^{cu}$.

To obtain that $\Lambda_g\en U_\alpha$, notice that if $\rho_1$ is small enough, then $f$ and $g$ are arbitrarily $C^0$-close and that the map $h \mapsto \Lambda_h$ is upper-semicontinuous for the Hausdorff topology: i.e. given a neighborhood $V$ of $\Lambda_h$, for $h'$ in a small $C^0$-small neighborhood of $h$, one has that $\Lambda_{h'} \subset V$.
\end{proof}

\begin{obs}\label{Remark-LambdaInUalpha}
From now on, we will choose $U_\alpha \subset U_\theta$ and $\rho_1$ will be chosen so that $\Lambda_g \subset U_\alpha$ and $\cW^\sigma_{g,\rho_2}(x)\subset U_\alpha$ for any $x\in \Lambda_g$ and $\sigma=s,u$.
\end{obs}

As a consequence of Proposition \ref{Prop-NarrowingCones2} we obtain:

\begin{lema}\label{Lema-NotChangesDistance}
If $\rho<\rho_1$ then, for every $x \in\Lambda_g$ which belong to a cube $C^i_{\rho/3}$, the connected component $\cD^u$ of $\cW^u_g(x)\cap C^i_{\rho}$ containing $x$ satisfies the following property.

In the $\psi^i_\rho$-coordinates, the projection of $\cD^u$ on the $\RR^s\times \{0\}^c\times \RR^u$-plane is the graph of a $2\alpha$-Lipschitz map $[0,1]^u\to [0,1]^s$.
\end{lema}
\begin{proof}[\sc Proof]
{ Since $\cD^u$ is tangent to $\cE^{u}\subset \cE^u_{1/10}$, any vector $v=v^s+v^c+v^u$ tangent to $\cD^u$
satisfies $\|v^c\|\leq \frac 1 2 \|v^u\|$.} Since $\cE^{cu}_\alpha$ is preserved by $Dg$ the disk $\cD^u$ is tangent to $\cE^{cu}_\alpha$
{ hence $\|v^s\|\leq \alpha\|v^c+v^u\|\leq 2\alpha \|v^u\|$.}
\end{proof}


\subsection{Breaking the joint integrability}\label{ss:breakjoint}
In this subsection we show that if there are points in $M$ whose local unstable manifolds satisfy a certain configuration with respect to the rectangles $R_{i,j,k}$, then their unstable manifolds verify the conclusion of Theorem \ref{Thm-PerturbationResult}. First we prove this in coordinates, and then apply the results in the previous subsection to conclude the same for points in the manifold.

\begin{prop}\label{prop-variationLip}
For any $\theta,\eta,\kappa,{ \Delta}>0$, with { $\theta< \frac{\kappa}{1000 d2^d\Delta}$},
the following property { holds}:  In the coordinates $\RR^s \oplus \RR^c \oplus \RR^u$, let $\vartheta_1$ be the graph of a $\theta$-Lipschitz map $[0,\eta]^u\to [0,1]^{s+c}$  and let $ \vartheta_2, \vartheta_3$ be the graphs of {$\theta$}-Lipschitz maps ${ [0,10.2^{d}\eta\Delta]^s}\to [0,1]^{u+c}$ such that $\vartheta_2, \vartheta_3$ intersect $\vartheta_1$. Then there is no pair of points in $\vartheta_2 \cup \vartheta_3$ whose second coordinate differ by more than $\kappa \eta/20$.
\end{prop}

\begin{proof}[\sc Proof] The second coordinate of {each} graph { $\vartheta_i$} has variation of at most
{ $10d2^{d}\eta\Delta\theta$}. Then, if one chooses { $\theta< \frac{\kappa}{1000 d2^d\Delta}$} the conclusion is verified.
\end{proof}

\begin{cor}\label{cor-nojointint}
For any $\theta,\eta,\kappa,{ \Delta}>0$, with { $\theta< \frac{\kappa}{1000 d2^d\Delta}$}, there is $\rho_2 \in (0,\rho_1)$ such that if $\rho<\rho_2$, if $x,y\in \Lambda_g$ belong to a cube $C^i_{\rho/3}$
{ with $d(x,y)\leq 2^{s+u}\rho \eta \Delta$}
and if there exists $R_{i,j,k}$ such that
\begin{itemize}
\item the connected component $\cD^u(x)$ of $W^u_{g}(x)\cap C^i_\rho$ containing $x$ intersect no rectangle $R_{i,j,k'}$ in $T_{i,j}$ and is contained in $\Lambda_g$, \item the connected component $\cD^u(y)$ of $\cW^u_{g}(y)\cap C^i_\rho$ containing $y$ intersects $\frac 1 2 R_{i,j,k}$,
\end{itemize}
then there exists $x'\in \cD^u(x)\cap T_{i,j}$ such that the connected component of $\cW^s_g(x')\cap C^i_{\rho}$ containing $x'$ is the graph of a map $[0,1]^s\to [0,1]^{c+u}$ in the coordinates $\psi^i_\rho$ and does not meet $\cD^u(y)$.
\end{cor}
\begin{proof}[\sc Proof] Working in charts $\psi^i_\rho$ and using that the first $N$ preimages of the slice $T_{i,j}$ are disjoint from any other slice, one can use Proposition \ref{Prop-NarrowingCones} to obtain that
\begin{itemize}
\item $\cD^u(x)\cap T_{i,j}$ is the graph $\vartheta_1$ of a $\theta$-Lipschitz map  in the slice $T_{i,j}$ over the unstable coordinates,
\item $\cD^{u}(y)\cap T_{i,j}$ is the image by the elementary perturbation $h_\kappa^\eta$ of the graph of a $\theta$-Lipschitz graph  in the slice $T_{i,j}$ over the unstable coordinates.
\end{itemize}

Since $h_\kappa^\eta$ is $\kappa$-$C^1$-close to the identity and since $(h_\kappa^\eta)^{-1}(\cD^u(y))$ is the graph of a $\theta$-Lipschitz map, using Lemma \ref{Lema-ElementaryPerturbation} one obtains that the disk $\cD^u(y) \cap T_{i,j}$ has two points { $y'_2,y'_3$} whose center coordinate differ by at least $\frac{\kappa \eta}{10} - \kappa \eta \theta$, which is larger than $\frac{\kappa \eta}{20}$.
{ Moreover $y'_2,y'_3$ are at distance larger than $\rho\eta/4$ to the boundary of $T_{i,j}$.}

{ If one assumes by contradiction that the conclusion of the corollary is not satisfied,
the connected component of $\cW^s_g(y'_2)\cap C^i_{\rho}$ containing $y'_2$ intersects $\vartheta_1$ at a point $x'_2$.
Similarly there exists $x'_3\in \vartheta_1$ associated to $y'_3$.
Since $d(x,y)\leq 2^{s+u}\rho \eta \Delta$ and from Proposition \ref{Prop-NarrowingCones},
there exists $a\in [0,1]$ such that $x'_2,y'_2$ (resp. $x'_3,y'_3$)
belong to the graph $\vartheta_2$ (resp. $\vartheta_3$) of a $\theta$-Lipschitz function $[0,10.2^{d}\eta\Delta]^s\to [0,1]^{u+c}$.}

Now one can apply Proposition \ref{prop-variationLip} to obtain the { contradiction}.
\end{proof}

Now, the rest of the proof consists in being able to guarantee that for \emph{every} pair of points $x,y \in \Lambda_g$ in the same stable manifold at distance in $(r,r')$ there exists a forward iterate in a $C^i_{\rho/3}$ cube having the configuration given by the previous corollary. This will allow us to conclude.


\subsection{Coverings with bounded geometry}\label{ss:coverings}

Let $F$ be a subset of $U$ and $\xi >0$. We say that $F$ is a $\xi$-\emph{covering} of size $\rho$ of $U$ if the following holds:

\begin{itemize}
\item If $x,y \in U$ are at distance smaller than $\rho/4$ then there exists $z\in F$ such that $x,y \in C^z_{\rho/3}$.

\item For every $\eps>0$ and for every ball $B$ of radius $\eps$ intersecting $U$ we have that:
$$ \# \{ x \in F \ : \ C^x_{2\rho} \cap B \neq \emptyset \} \leq \xi \max \bigg\{1, \left(\frac{\eps}{\rho}\right)^d \bigg\} $$
\end{itemize}

\begin{lema}\label{Lema-BoundedGeometry}
There exists $\xi>0$ such that for any small enough $\rho$, there
exists a finite set $F$ which is a $\xi$-covering of size $\rho$
of $U$.
\end{lema}

\begin{proof}[\sc Proof] We cover $U$ by cubes $C_1, \ldots, C_k$ of a sufficiently small scale so that there are charts $\psi_i : C_i \to [0,1]^s \times [0,1]^c \times [0,1]^u$. It is enough to cover each $C_i$ independently. This reduces the problem to the case where $U$ is the euclidean unit cube.

We therefore consider in $U$ the covering by  points in the $\rho/8$-lattice. For every pair of points at distance less than $\rho/4$ there is an element of the lattice such that its $\rho/3$-cube contains the pair of points.

Given a ball of radius $\eps$, it follows that its $2\rho$-neighbourhood has volume of order  $(\eps +2\rho)^d$ and bounded geometry. Therefore, this neighbourhood meets the lattice in a set with cardinal of the order $8 (\frac{\eps}{\rho} +2)$. This concludes.  \end{proof}

\begin{obs}\label{Remark-FixedValueXi}
In the previous lemma it is important that $\xi$ remains bounded as $\rho \to 0$. In fact, the value of $\xi$ only depends on the dimension of $M$.
\end{obs}

\begin{lema}\label{Lema-BoundedIntersectionIterates}
For every $N, \Delta, \xi>0$ there exist an integer $I:= I(N,\Delta,\xi)$ and $\rho_3>0$ with the following property. If $h:M \to M$ is a $C^1$-diffeomorphism with $\sup_{x\in U} \|D_x h\|< \Delta$  and if $F$ is a $\xi$-covering of size $\rho<\rho_3$ of $U$ then, for every $x \in F$, one has
$$ \# \{ y \in F \ : \ C^y_{2\rho} \cap  (C^x_{\rho} \cup h(C^x_\rho) \cup \ldots \cup h^N(C^x_\rho)) \neq \emptyset \} < I . $$
\end{lema}

\begin{proof}[\sc Proof]  If $\rho$ is small enough we know that $h^i(C^x_{\rho})$ is contained in a ball of radius $\Delta^i \rho$. Choosing $I \geq  (N+1) \xi  \max\{1,\Delta^{Nd}\}$ one obtains the desired bound. \end{proof}


\subsection{Wandering slices}\label{ss:wandering}

In this section we prove the following proposition which can be compared to Lemma 2.3 of \cite{DW}.

\begin{prop}\label{Prop-WanderingRegions}
Given $N>0$, there exists $\hat \eta:= \hat \eta(N,\Delta, \xi)>0$ and $\rho_4 \in (0,\rho_3)$ such that for every $\xi$-covering $F$ of $U$ of size $\rho< \rho_4$ and for every $x \in F$ there exists a slice $\hat T_x$ of width $\hat \eta$ inside $C_\rho^x$
intersecting $C^x_{\rho/2}$ and such that
the sets in $\{f^{-\ell}(\hat T_x) \ : \ x\in F \ , \ 0 \leq \ell \leq N \}$ are pairwise disjoint.
\end{prop}

We will call \emph{wandering slices} the sets $\hat T_x$ with $x\in F$ .

\begin{proof}[\sc Proof]
Consider $\hat \eta = \frac{1}{8I\Delta^N}$, where $I$ is chosen according to Lemma~\ref{Lema-BoundedIntersectionIterates}.
By choosing $\rho_4$ small enough one can work as if $f$ were linear inside each cube of the $\xi$-covering modulo a small error.

We perform an induction argument on the set $F= \{x_1, \ldots, x_m \}$.
Let $0\leq k\leq m$. Assume that we have chosen $\hat T_{x_i}$ for $i \leq k$ satisfying the above properties.
(When $k=0$ there is no condition.) We have to build $\hat T_{x_{k+1}}$.

Let $\cS_{k+1}$ be the set of slices of the form
$[0,1]^{s+c} \times (u_{x_{k+1}} + [-\frac{\hat \eta}{2},\frac{\hat \eta}{2}]^u)$,
with $u_{x_{k+1}} \in \hat \eta(\ZZ^u \cap [\frac 1 {3\hat \eta}, \frac 2 {3\hat \eta}]^u)$.
By definition these slices intersect $C^x_{\rho/2}$.
Moreover, there are at least $(4\hat \eta)^{-u}$ such slices.  We want to choose $\hat T_{x_{k+1}}$
in the set $\cS_{k+1}$.

From Lemma~\ref{Lema-BoundedIntersectionIterates},
one sees that among the $\rho$-cubes centered at points of $F$, there are at most $I$ which intersect
$C^{x_{k+1}}_\rho\cup f^{-1}(C^{x_{k+1}}_\rho)\cup\dots\cup f^{-N}(C^{x_{k+1}}_\rho)$.
In particular at most $I$ slices $\hat T_{x_i}$, $1\leq i\leq k$, have an iterate by $f^\ell$, $0\leq \ell\leq N$
which intersects $C^{x_{k+1}}_\rho$.
Using the fact that at this scale all first $N$-iterates of $f$ are almost linear
and that $\sup_{x\in U} \|D_x f\|< \Delta$,
each iterate $f^\ell(\hat T_{x_i})$ (with $1 \leq \ell \leq N$) can intersect at most
$(2\Delta^N)^u$ slices in $\cS_{k+1}$.
Since $I\cdot (2\Delta^N)^u<(4\hat \eta)^{-u}$ by the choice of $\hat \eta$,
there exists $\hat T_{x_{k+1}}\in \cS_{k+1}$ with the desired properties.
\end{proof}


\subsection{A sparse section for the unstable direction}\label{ss:section}

In this subsection we prove a geometric/combinatorial result which prepares the choice of rectangles in the next subsection.

\begin{prop}\label{Prop-Colorings} There exist $L:= L(\Delta,s)$ and a family $\cC$ of tiles in $\RR^s \times [0,L]^u$
{ such that if $\alpha<1/(20L)$ then the following properties hold:}

\begin{itemize}
\item[(i)] Each tile $Q \in \cC$ is of the form $(a,b) + [0,1]^s \times [0,1]^u$ with $a\in \RR^s$ and $b \in \{0, \ldots, L-1\}^u$.

\item[(ii)] Given two tiles $Q,Q' \in \cC$  in $\RR^s \times (b + [0,1]^u)$ with $b \in \{0, \ldots, L-1\}^u$ we have that $d(Q,Q') >3^{s+u}\Delta$.

\item[(iii)] If $\cD$ is the graph of a {$2\alpha$}-Lipschitz function $[0,L]^u \to \RR^s$ then there exists $Q \in \cC$ such
that $\cD \cap \frac{1}{2}Q\neq\emptyset$.
\end{itemize}
\end{prop}

As before, if $Q= a+ [0,1]^s \times [0,1]^u$ we denote by $\frac{1}{2}Q$ the set $a + [\frac{1}{4},\frac{3}{4}]^s \times [\frac{1}{4},\frac{3}{4}]^u$ .

\begin{proof}[\sc Proof] Choosing large enough $L$ it is possible to construct a familly $\cC$ of tiles satisfying properties (i) and (ii) and with property

\begin{itemize}
\item[(iii)']  for every $x\in \RR^s$ there exists $Q \in \cC$ such that
$$\{x\} \times [0,L]^u \cap \frac 1 4 Q \neq \emptyset.$$
\end{itemize}

To see this, consider a finite covering $\cN$ of the ball of radius $10^{s+u}\Delta$ by tiles of side $\frac 1 8$. We can choose $\cN$ to have less than $(800^{s+u} \Delta)^s$ tiles.

Choose $L$ so that  $L \geq (800^{s+u} \Delta)^s$. For each $b \in \ZZ^u \cap [0,L-1]^u$ we associate a tile in $\cN$ so that every tile corresponds to at least one $b$.  For each such $b$, choose $a \in \RR^s$ to be the lower corner of the corresponding tile.

Let $\cC_b$  be the set of tiles of the form $(a + 5^{s+u}\Delta k, b) +  [0,1]^s\times [0,1]^u$ with $k \in \ZZ^s$. This ensures that for every $x\in \RR^s$ there exists $b \in \ZZ^u \cap [0,L-1]^u$ and $Q \in \cC_b$ such that $\{x\} \times [0,L]^u \cap \frac 1 4 Q \neq \emptyset$ giving property (iii)'. The
fact that it verifies properties (i), (ii) is direct from the definition of $\cC_b$.

If $\alpha { < \frac{1}{20L}}$ the {$2\alpha$}-Lipschitz disks as in property (iii) are at distance smaller than {$\frac 1 8$} from disks as in property (iii)', so property (iii) will also be verified.  \end{proof}

\begin{cor}\label{Cor-TwoDisksAndRcoloring}
In the conditions of Proposition \ref{Prop-Colorings}, let $\cD,\cD'$ be { $2\alpha$}-Lipschitz unstable disks and $\cD^s$ a { $1/2$}-Lipschitz stable disk intersecting $\cD$ and $\cD'$ at points $x,y$ such that $d_{s}(x,y)  \in [2^{s+u},2^{s+u}\Delta]$. Then, there is $b\in \ZZ^u \cap [0,L-1]^u$ such that $\cD$ intersects $\frac{1}{2}Q$ for some $Q \in \cC_b$ whereas $\cD'$ does not intersect any tile in $\cC_b$.
\end{cor}

\subsection{Choice of the rectangles}\label{ss:choicerect}
In this section we place the rectangles $R_{i,j,k}$ where the elementary perturbations are
supported.

Given $L, \hat \eta$ and a slice $\hat T_z$ in a $\rho$-cube $C_\rho^z$ which in coordinates $\psi^z_\rho$ is of the form $[0,1]^{s+c} \times [0,\hat \eta]^u + u_z$ for $u_z \in \{0\}^{s+c}\times [0,1-\eta]^u$, one considers the \emph{sub-slices} $T_{z,j}$ with $j = \ZZ^u \cap [0,L-1]^u$ of width $\eta:=\hat \eta/L$  and which are of the form $[0,1]^{s+c} \times [0,\eta]^u + u_{z,j}$ with $u_{z,j}= u_z + \eta j$.

\begin{prop}\label{prop-choicerect}
Given {$L$, $\alpha$ as in Proposition~\ref{Prop-Colorings}} and $\hat \eta$ there exists $\rho_5>0$ with the following property. If we consider $\rho < \rho_5$, a slice $\hat T_z$ of width $\hat \eta$ inside a $\rho$-cube $C_\rho^z$ and  the sub-slices $T_{z,j}$ of $\hat T_z$ of width $\eta$, then there exist rectangles $R_{z,j,k} \en T_{z,j}$ of width $\eta=\hat\eta/L$ satisfying:

If $\cD^u_1, \cD^u_2 \en C_{\rho}^z$ are graphs of functions $[0,1]^u \to [0,1]^{s+c}$
{ and $\cD^s \en C_{\rho}^z$ is the graph of a function $[0,1]^s \to [0,1]^{c+u}$} in the coordinates $\psi^z_\rho$ which verify:

\begin{itemize}
\item $\cD^u_1$ and $\cD^u_2$ are tangent to both $\cE^{cu}_\alpha$ and $\cE^u$, { and $\cD^s$ is tangent to $\cE^{s}$,
\item $\cD^u_1$ and $\cD^u_2$ intersect $\cD^s$ at points $x,y \in C_{\rho/3}^z$,
\item $d_{s}(x,y) \in [2^{s+u} \rho\eta,2^{s+u}\rho \eta \Delta],$}
\end{itemize}

\noindent then, there exists $j$ and $k$ such that inside $T_{z,j}$:

\begin{itemize}
\item The disk $\cD^u_1$ intersects the rectangles $\frac{1}{2}R_{z,j,k}$.
\item The disk $\cD^u_2$ does not intersect any of the rectangles $R_{z,j,k'}$.
\end{itemize}
\end{prop}

\begin{figure}[ht]
\begin{center}
\includegraphics[scale=0.4]{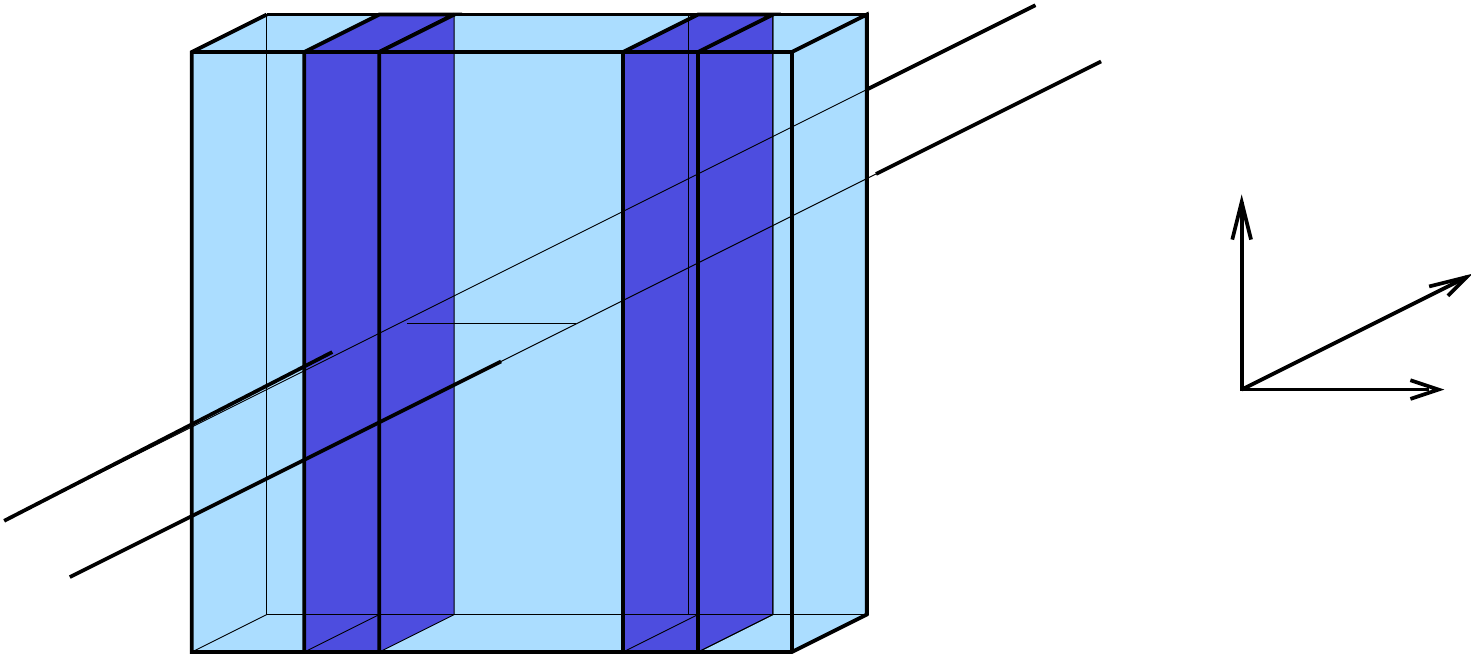}
\begin{picture}(0,0)
\put(-220,130){$R_{z,j,1}$}
\put(-160,130){$R_{z,j,2}$}
\put(-114,62){$T_{z,j}$}
\put(-280,40){$\cD_1$}
\put(-280,5){$\cD_2$}
\put(-44,83){$c$}
\put(-10,54){$s$}
\put(-10,75){$u$}
\end{picture}
\end{center}\end{figure}

\begin{proof}[\sc Proof] We use the familly of tiles $\cC$ given by Proposition \ref{Prop-Colorings} in each wandering slice to construct the rectangles. To do this, we expand the coordinates given by $\psi^z_\rho$ by a factor of $\frac{1}{\eta}$ and define the rectangles $Q \in \cC$
of the form $(a,b) + [0,1]^s \times [0,1]^u$ with $a \in [0, \frac{1}{\eta}-1]^s$. So, the rectangle $R$ associated to $Q$ is of the form

$$ R= (\psi^z_\rho)^{-1} \bigg(\bigg( \frac{a}{\eta} +[0,\eta]^s \bigg) \times [0,1]^c \times \bigg(\frac{b}{\eta} + [0,\eta]^u \bigg)\bigg) . $$

Notice that as the disks $\cD^{u}_1$ and $\cD^{u}_2$ are tangent to $\cE^{cu}_\alpha$, their projection onto the stable-unstable direction is the graph of a { $2\alpha$}-Lipschitz map from $[0,1]^u \to [0,1]^s$.
{ As the disk $\cD^{s}$ is tangent to $\cE^{s}\subset \cE^s_{1/10}$, their projection onto the stable-unstable direction is the graph of a $1/2$-Lipschitz map from $[0,1]^s \to [0,1]^u$.}
After a suitable change of scale, the property of the disks is a restatement of Corollary \ref{Cor-TwoDisksAndRcoloring}.
\end{proof}


\subsection{Choosing the perturbation}\label{ss:perturbation2}

The only thing missing to construct the perturbation is to determine the constants. These will be summarised in this section.

Fix a neighbourhood  $\cV$ of $f$ in $\Diff^1(M)$
and small constants $r,r't,\gamma$ as in the statement of Theorem~\ref{Thm-PerturbationResult}.
 As in the beginning of the section we fix neighbourhoods $U,U_1, \cU$  so that for any diffeomorphism $g\in \cU$ the maximal invariant set $\Lambda_g$ in $U_1$ is still partially hyperbolic with the same splitting. There are well defined cone-fields in $U$ which satisfy the properties of Proposition \ref{Proposition-Cones}.

We have $\Delta = \sup_{x\in M, g \in \cU} \{\|D_xg\|, \|D_xg^{-1}\|\}$, $s= \dim E^s$ and $d=\dim M=s+c+u$ fixed.
\medskip

\noindent
{\bf Fixing $\xi$.} We choose $\xi$ as in Lemma \ref{Lema-BoundedGeometry} which depends only on $d$.

\medskip

\noindent
{\bf Fixing $\kappa$.}  The neighborhoods $\cU$ and $\cV$ gives us a value of $\kappa$ and $\rho_0>0$ via Lemma \ref{Lema-ManyElementaryPerturbations} so that by applying disjoint elementary perturbations of size $\kappa$ in subsets of $\rho$-cubes ($\rho<\rho_0$) gives a diffeomorphism in $\cU \cap \cV$.

\medskip

\noindent
{\bf Fixing $L$ and $\theta$.} We choose $L:=L(\Delta,s)$ as in Proposition \ref{Prop-Colorings} and
{ $\theta< \frac{\kappa}{1000 d2^d\Delta}$}
(c.f. Corollary \ref{cor-nojointint}).

\medskip

\noindent
{\bf Fixing $N$.} We choose $N$ depending on $\theta$ via Proposition \ref{Prop-NarrowingCones}.

\medskip

\noindent
{\bf Fixing $\eta$ and $\hat \eta$.} { Proposition \ref{Prop-WanderingRegions} gives  $\hat \eta:= \hat \eta(N,\Delta,\xi)$ and $ \eta= \frac{\hat \eta}{L}$.}
\medskip

\noindent
{\bf Fixing $\alpha$.} We fix $\alpha<{ \frac 1 {20L}}$ using { Proposition \ref{Prop-Colorings}}.

\medskip

\noindent
{\bf Fixing $\rho$.}  The value of $\alpha$ gives a value $\rho_1>0$ via Proposition  \ref{Prop-NarrowingCones2} which bounds $\rho_2$ given by Corollary \ref{cor-nojointint}.  One fixes $\rho_3$ with Lemma \ref{Lema-BoundedIntersectionIterates} and $\rho_4 < \rho_3$ with Proposition \ref{Prop-WanderingRegions}. Finally we get $\rho_5$ via Proposition \ref{prop-choicerect}. We will also demand that $\rho < \rho_6 = \min\{r,t,\gamma\}$.

In summary, we fix $\rho < \min \{\rho_0, \rho_2, \rho_4, \rho_5, \rho_6\}$.

\medskip

\noindent
{\bf Realizing the perturbation.} We fix a $\xi$-covering $F=\{x_1, \ldots, x_m\}$ of $U$ given by Lemma \ref{Lema-BoundedGeometry} defining the cubes $C^i_\rho:=C^{x_i}_\rho$. Using Proposition \ref{Prop-WanderingRegions}, in each cube $C^i_\rho$ we have a wandering slice $\hat T_{i}$ which is disjoint from its first $N$-forward and backward iterates (as well as from the iterates of the other slices). The slice $\hat T_i$ decomposes as a union of subslices $T_{i,j}$ as in subsection \ref{ss:choicerect}.

By the choices of $\eta$ and $\hat \eta$ we can considering in each slice $T_{i,j}$ some rectangles $R_{i,j,k}$ using Proposition \ref{prop-choicerect}.

Once the rectangles are chosen, we obtain the diffeomorphism $g \in \cU \cap \cV$ as explained in subsection \ref{ss:perturbation1} by composing with elementary perturbations in each rectangle. It remains to check that $g$ verifies the conclusion of Theorem \ref{Thm-PerturbationResult}.

\subsection{Corroboration that the perturbation works}\label{ss:check}

The proof is by contradiction. Assume there are two points $x,y \in \Lambda_g$ such that $x \in \cW^s_g(y)$ with $d_s(x,y) \in [r,r']$ and such that for every $x' \in \cW^{u}_{g,t}(x)$ we have that $\cW^s_{g,\gamma}(x') \cap \cW^u_{g,\gamma}(y) \neq \emptyset$.

Iterating forward the points $x$ and $y$ they eventually become at distance $d_s(g^n(x),g^n(y)) \in [2^{s+u}\rho \eta, 2^{s+u}\rho \eta \Delta]$. As $2^{s+u}\rho \eta \Delta < \rho/4$ then $g^n(x)$ and $g^n(y)$ belong to a cube $C^i_{\rho/3}$ of the $\xi$-covering.

Forward iterates expand unstable manifolds, so, $g^n(\cW^u_t(x))$ and $g^n(\cW^u_\gamma(y))$ contain unstable disks $\cD^u_1$ and $\cD^u_2$ which cross $C^i_\rho$ and in particular the slices $T_{i,j}$. Moreover, these disks are tangent to $\cE^{cu}_\alpha$ and $\cE^u$ thanks to Lemma \ref{Lema-NotChangesDistance}.
{ Both points $g^n(x)$ and $g^n(y)$ belong to a stable disc $\cD^s$
which crosses $C^i_\rho$ and is tangent to $\cE^s$.}
So, we can apply Proposition \ref{prop-choicerect}.
It implies that the conclusion of Corollary \ref{cor-nojointint} holds.

In particular, there exists a point of $T_{i,j}$
- that we denote by $g^n(x')$ -
in the connected component $\cD^u(g^n(x))$ of $\cW^u(g^n(x))\cap C^i_\rho$
containing $g^n(x)$ with the following property:  the connected components $\cD^u(g^n(y))$ of $\cW^u(g^n(y))\cap C^i_\rho$
containing $g^n(y)$ and  $\cD^s(g^n(x'))$ of $\cW^s(g^n(x'))\cap C^i_\rho$ containing $g^n(x')$ do not intersect.

Since $\rho<t$, iterating by $g^{-n}$, one gets $x'\in \cW^u_t(x)$.
By our assumption, $\cW^s_{g,\gamma}(x') \cap \cW^u_{g,\gamma}(y)$ intersect at a (unique) point $y'$.
By Proposition~\ref{Prop-LocalProductStructure}, for each $0\leq k\leq n$, the distances
$d(g^k(x'),g^k(y'))$ and $d(g^k(y),g^k(y'))$ are smaller that
$2d(g^k(x'),g^k(y))$. In particular, the distances $d_s(g^n(x'),g^n(y'))$ and $d_u(g^n(y),g^n(y'))$ are smaller that
$\rho/2$.

Since $g^n(x')\in T_{i,j}$, it belongs to $C^i_{\rho/2}$.
As a consequence the $\rho/2$-neighbor\-hood of $g^n(x')$ in $\cW^s(g^n(x'))$ is contained in $\cD^s(g^n(x'))$.
One deduces that $g^n(y')\in \cD^s(g^n(x'))$.
Since $d_u(g^n(y),g^n(y'))\leq \rho/2$ and $g^n(y')\in C_\rho^i$,
one also gets $g^n(y')\in \cD^u(g^n(y))$.
Hence $\cD^s(g^n(x'))$ intersects $\cD^u(g^n(y))$, which contradicts the conclusion of Corollary \ref{cor-nojointint}.

\section{Proof of Theorem \ref{Theorem-MainGenerique}}\label{s.teogenerique}

In this section we will show how the pertubation result Theorem
\ref{Thm-PerturbationResult} implies Theorem
\ref{Theorem-MainGenerique} by a standard Baire argument.
First we need to show that the property obtained by the
perturbation is robust.

\begin{lema}\label{Lemma-RobustProperty}
Let $A\en M$ be a compact set and for each $C^1$-diffeomorphism
$f: M \to M$, let $\Lambda_f:=\bigcap_{n\in \ZZ}f^n(A)$ be the maximal invariant set in $A$.
Let $f_0$ such that $\Lambda_{f_0}$ is partially hyperbolic.
Then there exists an open $C^1$-neighborhood $\cU$ of $f$ and $r_0>0$ such that
for each $r<r'$ in $(0,r_0)$ and each $t,\gamma>0$,
the set of diffeomorphisms $f\in \cU$ satisfying the two following properties simultaneously is $C^1$-open:
\begin{itemize}
\item $\Lambda_f$ is partially hyperbolic.
\item For every $x,y \in \Lambda_f \cap \cW^s(x)$ with $d_s(x,y)
\in [r, r']$ such that $\cW^u_t(x) \en \Lambda_f$ there
exists $x' \in \cW^u_t(x)$ such that $ \cW^s_\gamma (x') \cap
\cW^u_\gamma (y) = \emptyset$.
\end{itemize}
\end{lema}
\begin{proof}[\sc Proof]
We choose a small open $C^1$-neighborhood $\cU$ of $f_0$ and $r_0>0$ such that
for every $f\in \cU$, the set $\Lambda_f$ is partially hyperbolic, and
moreover for each $x\in \Lambda_f$,
the strong stable manifold $\cW^s(x)$ contains a disc of radius larger
than $r_0$ for the distance $d_s$.

We then proceed by contradiction.
Assume otherwise, that there exist $f\in \cU$ satisfying the second property
and a sequence $g_n \to f$ in the $C^1$-topology and pairs of
points $x_n, y_n \in \Lambda_{g_n} \cap \cW^s_{g_n}(x)$ with
$d_s(x_n,y_n) \in [r, r']$ such that $\cW^u_{g_n,t}(x_n) \en
\Lambda_{g_n}$ but for every $x' \in \cW^u_{g_n,t}(x_n)$ we have
that $\cW^s_{g_n,\gamma}(x')\cap \cW^u_{g_n,\gamma}(y_n))\neq \emptyset$.

Using the continuity of stable manifolds (Theorem \ref{Teo-StableManifold})  and considering limits of $x_n,y_n$ we find two points $x,y \in
\Lambda_f \cap \cW^s(x)$ such that $d_s(x,y) \in [r,r']$ and
such that $\cW^u_t(x) \en \Lambda_f$ for which the second property does not
hold. This is a contradiction and concludes the proof of the
Lemma.\end{proof}

Now we are ready to prove Theorem \ref{Theorem-MainGenerique}:

\begin{proof}[\sc Proof of Theorem \ref{Theorem-MainGenerique}] Consider a countable
base of the topology of $M$ by open sets and let $\{O_k\}_{k\in
\NN}$ be the set of finite unions of elements in the base.

Let $\cA^1_k$ be the set of diffeomorphisms $f$ such that the
maximal invariant set $\Lambda_{k,f} = \bigcap_n f^n(\overline{O_k})$
of $\overline{O_k}$ is partially hyperbolic. Consider the set
$\cO^1_k= \cA^1_k \cup \overline{\cA^1_k}^c$. Since $\cA^1_k$ is
open (see Proposition \ref{Proposition-Cones}) we get that
$\cO^1_k$ is open and dense in $\Diff^1(M)$.

Let $T=\{(r,r',t,\gamma)\in (0,1)^4,\; r<r'\}$.
For each $\tau=(r,r',t,\gamma)\in T$,
we consider $\cA^2_{k,\tau} \en \cA^1_k$ to be the
$C^1$-interior of the set of $f \in \cA^1_k$ such that:
\begin{itemize}
\item[] For every pair of points $x,y \in \Lambda_{k,f} \cap \cW^s(x)$
such that $d_s(x,y) \in [r,r']$ and such that $\cW^u_{t}(x)
\en \Lambda_{k,f}$ then there exists $x' \in \cW^u_t(x)$ such that:

$$ \cW^s_\gamma(x') \cap \cW^u_\gamma(y) = \emptyset $$

\end{itemize}

Consider the sets $\cO^2_{k,\tau} = (\cA^2_{k,\tau} \cup
\overline{\cA^2_{k,\tau}}^c) \cap \cO^1_k $ which are again open
and dense by definition.

Let us show that the set $\cG= \bigcap_{k,\tau} \cO^2_{k,\tau}$ with
$k \in \NN$ and $\tau$ varying in a countable dense set of $T$ verifies the properties claimed in Theorem
\ref{Theorem-MainGenerique}. The set  $\cG$ is $G_\delta$-dense
by construction.

Let $f \in \cG$ and let $\Lambda$ be a partially hyperbolic set
for $f$ which is $\cW^u$-saturated. It holds that inside any
compact neighbourhood $V$ of $\Lambda$ there exists $k$ such that
$\Lambda \en O_k \en  \overline O_k \en \interior V$.

Consider a sufficiently small neighbourhood $O_k$ of $\Lambda$
and $\tau=(r,r',t,\gamma)\in T$ sufficiently close to $0$.
Using Theorem \ref{Thm-PerturbationResult}
and Lemma \ref{Lemma-RobustProperty},
we know that $f$ must
belong to the closure of $\cA^2_{k,\tau}$.
By construction this implies that $f$ belongs to the open set $\cA^2_{k,\tau}$.
This implies that, using
again Lemma \ref{Lemma-RobustProperty}, there exists $\delta>0$
which verifies the conditions in Theorem
\ref{Theorem-MainGenerique}. This concludes the proof.
\end{proof}

\section{Applications to minimal $u$-saturated partially
hyperbolic sets}\label{Section-Applications}

In this section we will prove Theorem \ref{Teo-FinitelyManyMinimal}. We first state some results we will use in the proof.

\subsection{Some preliminaries}
Let us fix $k\geq 2$. We denote by $B$ the open unit ball $B(0,1)$ in $\RR^k$,
by $D$ the $1$-codimensional closed disc $\overline{B(0,1/2)}\cap (\RR^{k-1}\times \{0\}$
and $\beta=\{0\}^{k-1}\times [-1/2,1/2]$.

We will use  the following result from topology.

\begin{lema}\label{l.top}
Let us consider two continuous maps
$\hat\beta \colon \beta \to \mathbb B$
and $\hat D \colon D \to B$ which are $1/8$-close to the inclusions for the $C^0$-distance.
Then the images of $\hat D$ and $\hat \beta$ intersect.
\end{lema}
\begin{proof}[\sc Proof]
Note that it is enough to prove it assuming that the maps $\hat \beta$ and $\hat D$ are smooth.
Indeed if one assumes by contradiction that the images of two continuous maps $\hat D$
and $\hat \beta$ do not intersects, one can approximate these maps by smooth ones in the $C^0$-topology,
which do not intersect either.

In the case $\hat \beta$ and $\hat D$ are smooth,
we build two smooth maps $\eta\colon S^1\to B$
and $\Sigma\colon S^{k-1}\to B$ as follows.

We decompose $S^1$ into four arcs
$I_1,\dots,I_4$ identified with $\beta$, and with disjoint interiors.
The map $\eta$ on $I_1$ is the identity and on $I_3$ is $\hat \beta$.
Moreover, the image of $I_2$ and $I_4$ under $\eta$ has diameter smaller than $1/8$.

We decompose the sphere $S^{k-1}$ into three submanifolds $D_1,D_2, A$,
glued along their boundaries such that
\begin{itemize}
\item $A$ is an annulus $(\partial D)\times [0,1]$,
\item $D_1,D_2$ are identified to $D$.
\end{itemize}
The map $\Sigma$ on $D_1$ is the identity and on $D_2$ is $\hat D$.
Moreover for each $x\in \partial D$, the image of $\{x\}\times [0,1]\subset A$
has diameter smaller than $1/8$.

One associated to the maps $\eta$ and $\Sigma$ and intersection number in $\ZZ/2\ZZ$
(see \cite[Chapter 2]{guillemin-pollack}), which is  invariant under homotopy,
hence vanishes since the ball is simply connected.

Note that the image of $\Sigma$ does not intersect $\eta(I_2\cup I_4)$
and the image of $\eta$ does not intersect $\Sigma(A)$.
Considering the restrictions of $\eta$ to $I_1,I_3$ and of $\Sigma$ to $D_1,D_2$,
one gets four intersection numbers, denoted by
$\mathcal I(I_i,D_j)$. Their sum vanishes (mod $2$).
By construction $\mathcal I(I_1,D_1)=1$.

In order to prove that the images of $\hat \beta$ and $\hat D$ intersect, it is enough to show that
$\mathcal I(I_3,D_2)=1$. Hence, it is enough to show that
$\mathcal I(I_1,D_2)=\mathcal I(I_3,D_1)=1 \mod 2$.

In order to compute $\mathcal I(I_1,D_2)$, one
build another map $\eta_0\colon S^1\to B$ which coincides with the identity on $I_1$
and such that $\eta_0(I_2\cup I_3\cup I_4)$ is disjoint from the image of $\Sigma$.
(This is possible since the image of $\Sigma$ is contained in the $1/8$-neighborhood of $D$
whose complement is connected.)
This gives $\mathcal I(I_1,D_2)=\mathcal I(I_1,D_1)=1\mod 2$.

In order to compute $\mathcal I(I_3,D_1)$, one
build another map $\Sigma_0\colon S^{k-1}\to B$ which coincides with the identity on $D_1$
and such that $\Sigma_0(A\cup D_2)$ is disjoint from the image of $\eta$.
(This is possible since the image of $\eta$ is contained in the $1/8$-neighborhood of $\beta$.)
This gives $\mathcal I(I_3,D_1)=\mathcal I(I_1,D_1)=1\mod 2$.
\end{proof}

We will use the following result which is a consequence of the main result in \cite{BC-Center}.

\begin{teo}\label{teo:BC}
Let $\Lambda$ be a partially hyperbolic compact invariant set for a diffeomorphism $f:M \to M$ with splitting $T_\Lambda M =E^{s} \oplus E^c \oplus E^u$ and assume that for every $x\in \Lambda$ one has that $\cW^s(x) \cap \Lambda = \{x\}$. Then there exists $S \en M$ a locally invariant embedded submanifold containing $\Lambda$ and tangent to $E^c \oplus E^u$ at every point of $\Lambda$.
\end{teo}

By \emph{locally invariant} we mean that there exists a neighborhood $V$ of $\Lambda$ in $S$ such that $f^{\pm 1}(V) \en S$; in particular, one can think as if the dynamics of $\Lambda$ is a lower dimensional dynamics (where the stable direction is not seen).

\begin{obs}
When the center direction is one-dimensional and $f$ is $C^1$-generic, it follows directly from the work of \cite{CPS} that when the hypothesis of Theorem \ref{teo:BC} are satisfied, then $\Lambda$ is uniformly hyperbolic. The proof we present here will not use this fact.
\end{obs}

For a point $x\in M$ we denote by $\cS(x)$ its \emph{stable set}
$$\cS(x):=\{y, d(f^n(x),f^n(y)\to 0 \text{ as } n\to +\infty\}.$$
We will use a further result which is a combination of Ruelle's inequality (see for example \cite[Chapter IV.10]{ManheLibro}) and an estimate of the size of stable manifolds for ergodic measures with large negative Lyapunov exponents. We refer the reader to  \cite[Chapter IV]{ManheLibro} for definitions of topological and metric entropy as well as Lyapunov exponents.

\begin{prop}\label{prop:largemanifold}
Let $f: M \to M$ be a $C^1$-diffeomorphism and $\Lambda \en M$ a compact $f$-invariant subset with a partially hyperbolic splitting $T_\Lambda M =E^s \oplus E^c \oplus E^u$ and $\dim E^c=1$. Consider a continuous cone-field $\cE^{cs}$ defined on a neighborhood of $\Lambda$
and whose restriction to $\Lambda$ is a continuous cone-field around $E^{s}\oplus E^c$.
Then, for every $h>0$ there exists $\eps>0$ such that if $K \en \Lambda$ verifies:
\begin{itemize}
\item $K$ is a compact $f$-invariant subset contained in a locally invariant submanifold $S \en M$ tangent to $E^c\oplus E^u$ at every point of $K$,
\item the topological entropy of $f|_K$ is larger than $h$,
\end{itemize}
\noindent then, there exists a point $x \in K$ such that its stable set $\cS(x)$ contains a $C^1$-disk of radius $\eps$ centered at $x$ and tangent to $\cE^{cs}$.
\end{prop}

\begin{proof}[\sc Proof] By the variational principle (\cite[Chapter IV.8]{ManheLibro}) there must be an ergodic  measure $\mu$ with entropy larger than $h$. Using Ruelle's inequality (\cite[Theorem 10.2 (a)]{ManheLibro}) applied to the restriction of $f^{-1}$ to $S$ and the fact that $K \en S$ one has that the center Lyapunov exponent of $\mu$ has to be smaller than $-h$. It follows from \cite[Section 8]{ABC} that there is a point $x$ in the support of $\mu$ with a large stable manifold tangent to $\cE^{cs}$. This concludes.
\end{proof}

\subsection{Proof of Theorem \ref{Teo-FinitelyManyMinimal}}

We consider $f_0 \colon M \to M$ in
the dense G$_\delta$ subset $\cG\subset \Diff^1$ given by Theorem \ref{Theorem-MainGenerique}.
Let $K$ be a partially hyperbolic set of $f_0$ with $\dim(E^c)=1$.
Let $\cU$ and $U$ be neighborhoods of $f_0$ and $K$ such that
for any $f\in \cU$, the maximal invariant set in $U$
is partially hyperbolic with $\dim(E^c)=1$.
Moreover, as in Section~\ref{ss.cones}, there exist invariant continuous cone fields $\cE^{cs}, \cE^{cu},\cE^u$ defined on $U$
around continuous bundles that extend $E^{cs},E^{cu},E^u$.
From Proposition~\ref{Prop-LocalProductStructure}, up to reduce the neighborhood $\cU$ of $f_0$,
there exists a compact neighborhood $U_K\subset U$ of $K$ and $\gamma_0$
such that there exists a local product structure for any $f\in \cU$
between points at distance smaller than $\gamma_0$ in the maximal invariant set $\Lambda_f$ of $f$ in  $U_K$.

Let $B^{cu}$ denote the open unit ball in $\RR\times \RR^u$
and $D^u=B^{cu}\cap (\{0\}\times \RR^u)$.
Up to reduce $U_K$ and $\cU$, there exists a finite collection open sets $V_1,\dots, V_\ell$
covering $U_K$, and for any $f\in \cU$, $x\in \Lambda_f\cap V_i$,
there is a $C^1$ embedding $\Psi_{i,x,f}\colon B^{cu}\to M$ and $\gamma\in (0,\gamma_0)$  with the following properties.
\begin{itemize}
\item for each $i$, the map $\Psi_{i,x,f}$ varies continuously with $f$ and $x$ for the $C^1$-topology,
\item $\Psi_{i,x,f}(D^u)$ is contained in the local unstable manifold of $x$ for $f$,
\item the image $\Psi_{i,x,f}(B^{cu})$ is tangent to $\cE^{cu}$, has diameter smaller than $\gamma_0$
and contains a ball of internal radius larger than $\gamma$ (for the induced metric).
\end{itemize}
Let us briefly explain the existence of these maps.
One first fixes a finite number of charts $V_i$ where all the bundles $TM, E^s, E^c, E^u$ are trivial.
Moreover, each $V_i$ is identified with an open connected set of $\RR^d$
so that for each $x\in V_i\cap \Lambda_f$, one can extend $E^c_x$ as a constant bundle over $V_i$.
The stable manifold theorem (c.f. Theorem \ref{Teo-StableManifold}) provides for each chart $V_i$,
a family of local unstable plaques that
can be parametrized continuously by $D^u$.
Each map $\Psi_{i,x,f}$ is now defined on $D^u$ and can be extended to $B^{cu}$
by flowing along a unit vector field tangent to the constant bundle $E^c_x$.

Up to reduce $\gamma>0$, for any $f\in \cU$, $x\in V_i\cap \Lambda_f$ and $y\in \Lambda_f$
that is $\gamma$-close to $x$,
one can define the projection
$\Pi^{ss}_{i,x,f}(y):=\cW^s_{\gamma_0}(y)\cap \Psi_{i,x,f}(B^{cu})$ which exists by Proposition~\ref{Prop-LocalProductStructure}.

Let $\Delta$ be a constant larger than $\max_{x\in M} \{\|D_xf\|, \|D_xf^{-1}\|\}$ for every $f\in \cU$.
Let $r'>\Delta.r>r>0$ and $t>0$ be much smaller than $\gamma$.
Theorem \ref{Theorem-MainGenerique} and Lemma \ref{Lemma-RobustProperty}
imply that, after reducing $\cU$, there exists the following property holds for every $f\in \cU$
and every $f$-invariant partially hyperbolic $\cW^u$-saturated set $\Gamma\subset U_K$:

\begin{itemize}
\item[(*)] There exists $\delta>0$ such that if $x,y \in \Gamma$ satisfy $y \in \cW^s(x)$ and $d_s(x,y) \in (r, r')$,
then there exists $x'\in \cW^u_t(x)$ such that $$ d(\cW^u_\gamma(y),\cW^s_\gamma(x))>\delta.$$
\end{itemize}

We want to show that such a saturated set $\Gamma$ contains at most finitely many \emph{minimal $\cW^u$-saturated} sets
(i.e. which are minimal for the inclusion among compact, $f$-invariant and $\cW^u$-saturated non-empty sets).

The following easy property will be important:

\begin{lema}\label{l.disjstable}
Let $\Lambda$ and $\Lambda'$ be minimal $\cW^u$-saturated sets such that there exists $x\in \Lambda$ such that $\cS(x) \cap \Lambda' \neq \emptyset$. Then $\Lambda=\Lambda'$.
\end{lema}

\begin{proof}[\sc Proof] Since the intersection of $f$-invariant $\cW^u$-saturated sets is $f$-invariant and $\cW^u$-saturated, we get that two different minimal $\cW^u$-saturated sets are either disjoint or coincide. Therefore, if $\Lambda' \neq \Lambda$ it follows that $d(\Lambda, \Lambda') > 0$. Using $f$-invariance, one deduces that $\cS(x) \cap \Lambda' = \emptyset$ forall $x\in \Lambda$.
\end{proof}

We distinguish between two types of minimal subsets which by definition cover all possibilities:

\begin{itemize}
\item  A \emph{Lower dimensional minimal set} is a minimal
$\cW^u$-saturated set $\Lambda$ such that for every $x\in \Lambda$
one has that $\cW^s(x) \cap \Lambda= \{x\}$.

\item  A \emph{Minimal set with strong connection} is a minimal set
$\Lambda$ for which there exists points $x,y \in \Lambda$
satisfying $y \in \cW^s(x)$.
\end{itemize}

We first show:

\begin{prop}\label{p.bigsizebasin}
In $\Gamma$ there are at most finitely many minimal $\cW^u$ saturated sets with a strong connection.
\end{prop}
\begin{proof}[\sc Proof] Assume by contradiction that there are infinitely many disjoint minimal $\cW^u$-saturated sets $\Lambda_n \en \Gamma$
admitting a strong connection.

As the sets $\Lambda_n$ are invariant, we can consider by iterating the connection, that inside $\Lambda_n$ there are points $x_n,y_n$ such that $y_n \in \cW^s(x_n)$ and $d_s(x_n,y_n) \in [r, \Delta r]$. Also, by (*), there is a point $z_n \in \cW^u_{t}(y_n)$ such that $d(\cW^s_\gamma(z_n), \cW^u_\gamma(x_n)) >\delta$.

Choosing a subsequence if necessary, we can assume that $\Lambda_n \to \Lambda \en \Gamma$ in the Hausdorff topology, that $x_n \to x \in \Lambda$, that $y_n \to y \in \Lambda$ and $z_n \to z \in \cW^u_t(y) \en \Lambda$. The set $\Lambda$ is also $\cW^u$-saturated (maybe not minimal) and it also has a strong connection since $y \in \cW^s(x)$ and $d_s(x,y) \in [r,\Delta r]$. These limit properties follow from the continuous variation of strong manifolds (Theorem \ref{Teo-StableManifold}).

Consider $V_i$ containing $x$ (and the $x_n$ for $n$ large),
the center-unstable disk $\cD^{cu}(x):=\Psi_{i,x,f}(B^{cu})$
and the unstable disc $\cD^{u}(x):=\Psi_{i,x,f}(D^{u})$ around $x$.
The unique point $\{w\}= \cW^s_\gamma(z) \cap \cD^{cu}(x)$ verifies that $d(w, \cD^u(x))\geq \delta$, as limit of the property (*).

Note that $\cD^{cu}(x)\setminus \cD^u(x)$ has two connected components $B^+,B^-$.
We can assume without loss of generality that $w \in B^-$.
Let $w_n$ be the points $\cW^s_\gamma(z_n) \cap \cD^{cu}(x)$. By continuity, for $n$ large enough, the points $w_n$ belong to $B^-$ as $w_n \to w$.

Let us consider an arc connecting $z$ to $y$ in $\cW^u_t(y)$.
Projecting by $\Pi^{ss}_{i,x,f}$, one gets an arc $\hat \beta_0$ in $\cD^{cu}(x)$ joining $w$ to $x$.
Since $t,r'$ have been chosen small compared to $\gamma$, this arc lifted by $\Psi_{i,x,f}$ has diameter smaller than $1/8$.
It can be concatenated with two closed arcs $\hat \beta_{-1},\hat \beta_1$ to produce an extended arc $\hat \beta$
in $\cD^{cu}(x)$ such that
\begin{itemize}
\item the lift $\Psi_{i,x,f}^{-1}(\hat \beta)$ is $1/8$ close to $\beta:=[-1/2,1/2]\times \{0\}^u$,
\item $\hat \beta_{-1}$ does not intersect $\cD^u(x)$,
\item $\hat \beta_1$ is contained in $B^+\cup \cD^u(x)$.
\end{itemize}

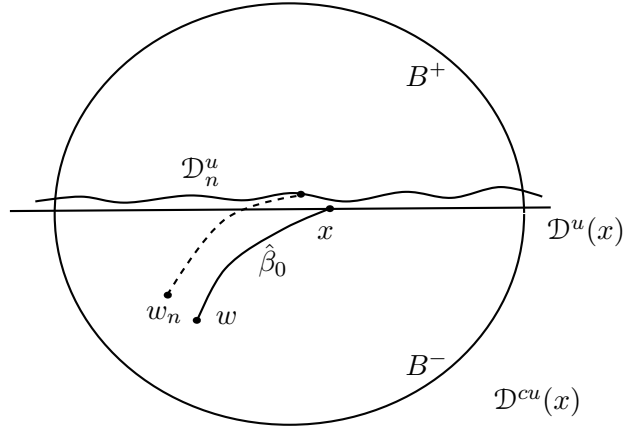
\begin{figure}[ht]
\vspace{-0.5cm}
\begin{center}
\includegraphics[scale=0.45]{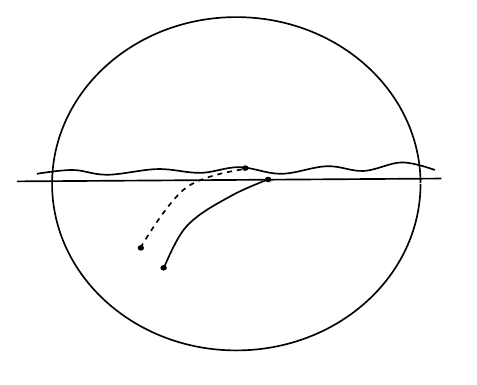}
\begin{picture}(0,0)
\put(-35,15){$\cD^{cu}(x)$}
\put(-75,35){$B^-$}
\put(-130,144){$B^+$}
\put(-79,76){$x$}
\put(-159,90){$\cD^{u}_n$}
\put(-95,55){$\hat \beta_0$}
\put(-72,125){$\hat \beta_{1}$}
\put(-95,20){$\hat \beta_{-1}$}
\put(-100,40){$w$}
\put(-120,37){$w_n$}
\put(-5,75){$\cD^u(x)$}
\end{picture}
\end{center}
\vspace{-0.5cm}
\caption{The stable manifolds of the minimal sets must intersect.\label{fig:intersection}}
\end{figure} 

Let $D^u_{1/2}$ denote the one-codimensional disc in $B^{cu}$ which is the
intersection of $D^u$ with the ball of radius $1/2$.
For each $n$, one considers the disc $\Psi_{i,x_n,f}(D^u_{1/2})$ and its projection $\cD^u_n$ by
$\Pi^{ss}_{i,x,f}$ on $\cD^{cu}(x)$.
As $n$ goes to $+\infty$, the lifts $\Psi^{-1}_{i,x,f}(\cD^u_n)$ converge in the $C^0$-topology to the inclusion of $D^u_{1/2}\subset B^{cu}$.
By Lemma \ref{l.top} for $n$ large enough, the disc $\cD^u_n$ intersects the arc $\hat \beta$.

The arc $\hat \beta_{-1}$ does not intersect $\cD^u$, and thus does not intersect $\cD^u_n$, for $n$ large enough.
There are two cases.
\begin{itemize}
\item If $\cD^u_n$ intersects $\hat \beta_0$, this implies by construction that the stable set of $\Lambda_n$
intersects the stable set of $\Lambda$.
\item If $\cD^u_n$ intersects $\hat \beta_{1}$, then there is a point $\zeta_n$ of $\cD^u_n$ in $B^+\cup \cD^u$.
There exists a connected set in the projection by $\Pi^{ss}_{i,x,f}$
of $\cW^u_t(y_n)$ joining $w_n$ to $\cD^u_n$.
Hence the projection of $\Lambda_n$ by $\Pi^{ss}_{i,x,f}$ contains a connected set in $\cD^{cu}(x)$
joining $w_n$ to $\zeta_n$. Since $w_n\in B^-$, this implies that the stable set of $\Lambda_n$
intersects the stable set of $\Lambda$.
\end{itemize}
Both cases give a contradiction with Lemma \ref{l.disjstable}.
\end{proof}

\begin{prop}\label{p.bigstableforsurface}
There exists $\eps>0$ such that if $\Lambda$ is a lower dimensional minimal $\cW^u$-saturated set of $\Gamma$, then $\Lambda$ contains
a point $x$
whose stable set $\cS(x)$ contains a $C^1$-disk of radius $\eps$ centered at $x$ and tangent to $\cE^{cs}$.
\end{prop}

\begin{proof}[\sc Proof] Using Theorem \ref{teo:BC} and Proposition \ref{prop:largemanifold} it is enough to show that there is a uniform lower bound on the topological entropy for each such minimal $\cW^u$-saturated set.

This follows from the following argument\footnote{A similar but sharper argument can be found in \cite{Newhouse-Volume}.}:
let $\gamma\in (0,\gamma_0)$ be a small constant and let $\eta>0$ much smaller than $\gamma$.
Consider a finite covering $\cC$ of $\Gamma$ by balls of radius $\eta$. There exists $k_0$
such that the image by $f^{k_0}$ of any disc tangent to $\cE^{u}$ centered at a point of $\Gamma$ and of radius $\eta$
contains a disc of radius $\gamma$.

Let $\Lambda$ be a minimal $\cW^u$-saturated set of $\Gamma$.
Any disk $D$ contained in $\Lambda$, tangent to a $E^u$ and of diameter $\gamma$
contains at least two disks of radius $\eta$ in different balls of the covering $\cC$.
Therefore, inside $D$ one has that in $k_0$ iterates we duplicate the number of $\eta$-separated orbits and therefore the entropy of $f$ in $\Lambda$ is larger than $\frac{1}{k_0}\log 2$ (independent of $\Lambda$).
\end{proof}

We now conclude the proof of Theorem \ref{Teo-FinitelyManyMinimal} by contradiction.
Let us assume that there exists an infinite sequence of minimal $\cW^u$-saturated set $(\Lambda_n)$in $\Gamma$.
From Proposition~\ref{p.bigsizebasin}, we can assume that they are lower dimensional.
From Proposition~\ref{Prop-LocalProductStructure} and Proposition \ref{p.bigstableforsurface},
there exists $n\neq m$ such that $\Lambda_n$ intersects the stable set $\cS(x_m)$ of some point $x_m\in \Lambda_m$.
This contradicts Lemma \ref{l.disjstable}.


\bigskip

\small
\noindent
\emph{Sylvain Crovisier}\\
{CNRS - Laboratoire de Math\'ematiques d'Orsay (UMR 8628),\\
Universit\'e Paris-Sud, 91405 Orsay, France.}\\
\texttt{http://www.math.u-psud.fr/$\sim$crovisie}
\bigskip

\noindent
\emph{Rafael Potrie}\\
{CMAT, Facultad de Ciencias, Universidad de la Rep\'ublica, Uruguay}.\\
\texttt{www.cmat.edu.uy/$\sim$rpotrie}
\texttt{rpotrie@cmat.edu.uy}
\bigskip

\noindent
\emph{Mart\'in Sambarino}\\
{CMAT, Facultad de Ciencias, Universidad de la Rep\'ublica, Uruguay}.\\
\texttt{samba@cmat.edu.uy}

\end{document}